\documentclass[a4paper,reqno]{amsart}
\usepackage{geometry}
\usepackage{mathrsfs}
\usepackage{syntonly}
\usepackage{amsmath}
\usepackage{amsthm}
\usepackage{amsfonts}
\usepackage{amssymb}
\usepackage{latexsym}
\usepackage{amscd,amssymb,amsopn,amsmath,amsthm,graphics,amsfonts,mathrsfs,accents,enumerate,verbatim,calc}
\usepackage[dvips]{graphicx}
\usepackage[colorlinks=true,linkcolor=blue,citecolor=blue]{hyperref}
\input xy
\xyoption{all}

\usepackage{color}

\numberwithin{equation}{section}

\theoremstyle{plain}

\newtheorem{thm}{Theorem}[section]
\newtheorem{cor}[thm]{Corollary}
\newtheorem{lem}[thm]{Lemma}
\newtheorem{prop}[thm]{Proposition}

\newtheorem{defn}[thm]{Definition}
\newtheorem{exm}[thm]{Example}

\newtheorem{rem}[thm]{Remark}

\newdir{ >}{{}*!/-5pt/\dir{>}}

\newcommand{\Fac}{\operatorname{Fac}\nolimits}

\newcommand{\Hom}{\operatorname{Hom}\nolimits}

\renewcommand{\Im}{\operatorname{Im}\nolimits}

\newcommand{\Ext}{\operatorname{Ext}\nolimits}

\newcommand{\id}{\operatorname{id}\nolimits}

\renewcommand{\mod}{\mathsf{mod}\hspace{.01in}}
\newcommand{\Cone}{\operatorname{Cone}\nolimits}
\newcommand{\CoCone}{\operatorname{CoCone}\nolimits}

\newcommand{\M}{\mathcal M}

\newcommand{\B}{\mathcal B}
\newcommand{\uB}{\underline{\B}}
\newcommand{\oB}{\overline{\B}}
\newcommand{\U}{\mathcal U}
\newcommand{\V}{\mathcal V}

\newcommand{\W}{\mathcal W}

\newcommand{\G}{\mathcal G}

\newcommand{\T}{\mathcal T}

\newcommand{\D}{\mathcal D}

\newcommand{\N}{\mathcal N}

\newcommand{\X}{\mathscr X}
\newcommand{\Y}{\mathscr Y}

\newcommand{\C}{\mathcal C}

\newcommand{\EE}{\mathbb E}

\newcommand{\svecv}[2]{\left(\begin{smallmatrix}
      #1 \\
      #2
    \end{smallmatrix}\right)}

\newcommand{\svech}[2]{\left(\begin{smallmatrix}
      #1 & #2
\end{smallmatrix}\right)}

\renewcommand{\emph}{\textit}
\renewcommand{\phi}{\varphi}

\newcommand{\add}{\mathsf{add}\hspace{.01in}}
\begin{document}

\title{Hereditary cotorsion pairs on extriangulated subcategories}\footnote{Yu Liu was supported by the Fundamental Research Funds for the Central Universities (Grant No. 2682018ZT25) and the National Natural Science Foundation of China (Grant No. 11901479). Panyue Zhou was supported by the National Natural Science Foundation of China (Grant Nos. 11901190 and 11671221), and by the Hunan Provincial Natural Science Foundation of China (Grant No. 2018JJ3205), and by the Scientific Research Fund of Hunan Provincial Education Department (Grant No. 19B239).}
\author{Yu Liu and Panyue Zhou}
\address{School of Mathematics, Southwest Jiaotong University, 610031 Chengdu, Sichuan, People's Republic of China}
\email{liuyu86@swjtu.edu.cn}
\address{College of Mathematics, Hunan Institute of Science and Technology, 414006 Yueyang, Hunan, People's Republic of China}
\email{panyuezhou@163.com}
%\thanks{The authors would like to thank Professor Bin Zhu for helpful discussions.}
\begin{abstract}
Let $\B$ be an extriangulated category with enough projectives and enough injectives. We define a proper $m$-term subcategory $\G$ on $\B$, which is an extriangulated subcategory. Then we give a correspondence between cotorsion pairs on $\G$, support $\tau$-tilting subcategories on an abelian quotient of $\G$ when $m=2$. If such $\G$ is induced by a hereditary cotorsion pair, then we give a correspondence between cotorsion pairs on $\G$ and intermediate cotorsion pairs on $\B$ under certain assumptions. At last, we study an important property of such extriangulated subcategory $\G$.
%Assume that $(\U,\V)$ is a hereditary cotorsion pair in $\B$ and $\mathcal W=\mathcal U\cap\mathcal V$.
%Our main results show that cotorsion pairs in $\mathcal B$ are in bijection with cotorsion pairs in $\mathcal G:=\Cone(\W,\W)$ (which is an extriangulated category) under some assumptions. We also show that
%cotorsion pairs in $\mathcal G$ are in bijection with support $\tau$-tilting subcategories in $\overline{\mathcal G}$ under some assumptions.
\end{abstract}
\maketitle

\section{Introduction}
Cotorsion pairs were first introduced  by Salce \cite{Sa} in the category of abelian groups, and then in abelian, exact and triangulated categories. Cotorsion pairs give a perspective of resolutions in these categories, and are related with many important homological notions in category theory and representation theory.
If one considers a cotorsion pair in a triangulated category, it becomes essentially
the analogue of a torsion pair. It gives a generalization of $t$-structure, co-$t$-structure and cluster tilting subcategory.

Let $\mathcal T$ be a triangulated category with a shift functor $\Sigma$.
There are many articles discussing the relation between co-$t$-structures and the silting theory. For example, it was shown in \cite{MSSS} that there exists a bijection between the bounded co-$t$-structures in $\T$ and the silting subcategories of $\T$.
If $(\X,\Y)$ and $(\X', \Y')$ are co-$t$-structures of $\T$, then $(\X', \Y')$
is called intermediate if $\X\subseteq \X'\subseteq \Sigma \X$. Iyama, J{\o}rgensen and Yang \cite{IJY} proved that intermediate
co-$t$-structures are in bijection with two-term silting subcategories, and also with
support $\tau$-tilting subcategories under certain assumptions.
%Later, Liu and Nakaoka \cite{LN} proved that the heart of a cotorsion pair on an extriangulated category is an abelian category.

Assume that $(\X,\Y)$ is a co-$t$-structure with coheart $\mathcal S=\Sigma\X\cap\Y$ and extended coheart $\mathcal C=\Sigma^2\X\cap\Y=\mathcal S\ast\Sigma\mathcal S$, then $\C$ is an extriangulated subcategory of $\T$.  Pauksztello and Zvonareva \cite{PZ} showed that there exists a bijection between intermediate co-$t$-structures in $\mathcal T$ and cotorsion pairs in $\C$. They also showed further that there exists a bijection between complete cotorsion pairs in $\C$ and
functorially finite torsion pairs in $\mod\mathcal S$ when it is Notherian.

The extriangulated category, introduced by Nakaoka and Palu \cite{NP}, is a simultaneous
generalization of exact category and triangulated category.
There are many examples of extriangulated categories which are neither exact triangulated categories nor triangulated categories, see \cite{HZZ,NP,ZZ}.
Nakaoka and Palu \cite{NP} defined the notion of a cotorsion pair on an extriangulated category, which is a generalization of cotorsion pairs on exact categories and on triangulated categories.

We want to investigate if there are similar bijections on exact categories, we also want to generalize the bijections given in \cite{PZ}. Hence we are considering the bijections under a more general setting.
Let $k$ be a field and $(\B,\EE,\mathfrak{s})$ be a Krull-Schmidt Hom-finite, $k$-linear extriangulated category.
We denote by $\mathcal P$ (resp. $\mathcal I$) the subcategory of projective (resp. injective) objects. When we say that $\C$ is a subcategory of $\B$, we always assume that $\C$ is full and closed under isomorphisms.

\begin{defn}\label{def2}
Let $\B'$ and $\B''$ be two subcategories in $\B$. Denote by $\Cone(\B',\B'')$ the subcategory
$$\{X\in \B \text{ }|\text{ } \textrm{there exists an}~ \text{ } \EE\text{-triangle } \xymatrix@C=0.5cm@R0.5cm{B' \ar[r] &B'' \ar[r] &X \ar@{-->}[r] &} \text{, }B'\in \B' \text{ and }B''\in \B''  \}.$$
Let $\Sigma \B'=\Cone(\B',\mathcal I)$ and $\Sigma^i \B'=\Cone(\Sigma^{i-1}\B', \mathcal I)$. %We write an object $D'$ in the form $\Sigma^i B'$ if it admits an $\EE$-triangle $\xymatrix@C=0.5cm@R0.5cm{B' \ar[r] &I \ar[r] &D' \ar@{-->}[r] & }$ where $I\in \mathcal I$.
\end{defn}

\begin{defn}
A subcategory $\W$ is called $n$-rigid if $\EE^{i}(\W,\W)=0$, $i=1,2,...,n$. A subcategory $\G$ is called a proper $m$-term subcategory induced by $\W$ if $\W$ is $n$-rigid, $m\leq n$ and
$$\G=\{X\text{ }|\text{ } \textrm{there exist}~ \text{ } \EE\text{-triangles} \xymatrix@C=0.5cm@R0.5cm{ X_{i+1} \ar[r] &W_i \ar[r] &X_i \ar@{-->}[r] &,}~\mbox{where}~i=1,...,m, X_{m+1},W_i\in \W,X_1=X  \}.$$
\end{defn}

%One typical example of two-term $n$-rigid  ($n\geq 2$) subcategory is induced by an $(n+1)$-cluster tilting subcategory on $\B$.
%According to \cite{LN}, when $\W$ is an $(n+1)$-cluster tilting subcategory, $\G$ becomes a two-term $n$-rigid subcategory.

Our first main result shows that cotorsion pairs on a proper $2$-term subcategory $\G$ induced by $\W$ correspond bijectively with support $\tau$-tilting subcategories on the ideal quotient $\G/\Sigma \W$ when it is abelian, giving a link between cotorsion pairs and $\tau$-tilting theory. %Note that $\rm\bf P(\mathcal A')$ denotes the subcategory $\{M\in \mathcal A'\text{ }|\text{ } \Ext^1_{\mathcal A}(M,\mathcal A')=0 \}$

\begin{thm}{\rm (See Theorem \ref{main2} for details)}
Let $\W$ be a $2$-rigid subcategory such that $\mathcal I\subsetneq \W$ and $\mathcal G=\Cone(\W,\W)$ be the proper $2$-term subcategory induced by $\W$. Assume that $\G/\Sigma \W$ is abelian. Then we have a one-to-one correspondence $$\Phi\colon (\X,\Y) \mapsto (\X/\Sigma \W)\cap(\Y/\Sigma \W)$$ from the first of the following sets to the second.
\begin{itemize}
\item[\rm (i)] Cotorsion pairs $(\X,\Y)$ in $\mathcal G$.
\smallskip

\item[\rm (ii)] Support $\tau$-tilting subcategories in $\G/\Sigma \W$.
\end{itemize}
\end{thm}

%Assume that $\W$ is a subcategory of $\B$. Denote $\Cone(\W,\W)$ by $\W_1$ and  $\Cone(\W_i,\W)$ by $\W_{i+1}$, where $\W_0:=\W$.

Since we don't have the concept of co-$t$-structure on the extriangulated category, we introduce the notion of hereditary cotorsion pair (which become a co-$t$-structure when $\B$ is tirangulated), see Definition \ref{hered}. Our second main result gives a bijection between cotorsion pairs on proper $(i+1)$-term subcategory $\W_i$ and intermediate cotorsion pairs on $\B$.

\begin{thm}{\rm (See Theorem \ref{main1} for details)}
Let $(\U,\V)$ be a hereditary cotorsion pair in $\B$ and $\mathcal W=\mathcal U\cap\mathcal V$. Denote $\Cone(\W,\W)$ by $\W_1$ and  $\Cone(\W_i,\W)$ by $\W_{i+1}$, where $\W_0:=\W$. Assume that
$$({^{\bot_1}}(\Sigma \V),\Sigma \V) \text{ and } ({^{\bot_1}}(\Sigma^i \V),\Sigma^i \V)$$
are also cotorsion pairs in $\B$. %where
%${^{\bot_1}}(\Sigma \V):=\{ M\in \B \text{ }|\text{ } \EE(M,\Sigma\V)=0\}$.
Then there exists a one-to-one correspondence from the first of the following sets to the second.
\begin{itemize}
\item[(a)] Cotorsion pairs $(\U',\V')$ in $\B$ such that $\U\subseteq \U'\subseteq {^{\bot_1}}(\Sigma^i \V)$.
\medskip

\item[(b)] Cotorsion pairs $(\X,\Y)$ in $\W_i$.
\end{itemize}
%The cotorsion pairs in {\rm (a)} are called \emph{intermediate}, and $\W_i$ is called an $(i+1)$-\emph{term subcategory}.
\end{thm}

Our third main result shows an important property of $\W_i$ $(i>0)$ under the settings of the above theorem.

\begin{thm}{\rm (See Theorem \ref{main5} for details)}
Let $(\U,\V)$ be a hereditary cotorsion pair in $\B$ and $\mathcal W=\mathcal U\cap\mathcal V$. If $\mathcal I\subsetneq\W$, then
$(\W_i, \Sigma^i\W_{j-i-1})$ is a hereditary cotorsion pair in $\W_j$ when $0<i<j$.
\end{thm}

%This article is organized as follows. In Section 2, we give some basic facts which will be used in the sequel.  In Section 3, we prove our first main
%result, see Theorem \ref{main2}. Moreover, we also give an application.  In Section 4, we prove our second main result, see Theorem \ref{main1}. We also give an example to explain it. In Section 5, we prove our third main result.

\section{Preliminaries}

In this article, let $k$ be a field and $(\B,\EE,\mathfrak{s})$ be a Krull-Schmidt Hom-finite, $k$-linear extriangulated category (see \cite{NP} for details of extriangulated categories).
 Let $\mathcal P$ (resp. $\mathcal I$) be the subcategory of projective (resp. injective) objects. When we say that $\C$ is a subcategory of $\B$, we always assume that $\C$ is full and closed under isomorphisms.

%\begin{defn}\label{def2}
%Let $\B'$ and $\B''$ be two subcategories of $\B$.
%\begin{itemize}
%\item[(a)] Denote by $\CoCone(\B',\B'')$ the subcategory
%$$\{X\in \B \text{ }|\text{ } \textrm{there exists an}~ \text{ } \EE\text{-triangle } \xymatrix@C=0.5cm@R0.5cm{ X \ar[r] &B' \ar[r] &B'' \ar@{-->}[r] &} \text{, }B'\in \B' \text{ and }B''\in \B'' \}.$$
%\item[(b)] Denote by $\Cone(\B',\B'')$ the subcategory
%$$\{X\in \B \text{ }|\text{ } \textrm{there exists an}~ \text{ } \EE\text{-triangle } \xymatrix@C=0.5cm@R0.5cm{B' \ar[r] &B'' \ar[r] &X \ar@{-->}[r] &} \text{, }B'\in \B' \text{ and }B''\in \B''  \}.$$
%\item[(c)] Let $\Omega \B'=\CoCone(\mathcal P,\B')$. We write an object $D$ in the form $\Omega B$ if it admits an $\EE$-triangle $\xymatrix@C=0.5cm@R0.5cm{D \ar[r] &P \ar[r] &B \ar@{-->}[r] &}$ where $P\in \mathcal P$.
%\smallskip
%
%\item[(d)] Let $\Sigma \B'=\Cone(\B',\mathcal I)$. We write an object $D'$ in the form $\Sigma B'$ if it admits an $\EE$-triangle $\xymatrix@C=0.5cm@R0.5cm{B' \ar[r] &I \ar[r] &D' \ar@{-->}[r] & }$ where $I\in \mathcal I$.
%\end{itemize}
%\end{defn}

\begin{defn}\cite[Definition 2.1]{NP}
Let $\U$ and $\V$ be two subcategories of $\B$ which are closed under direct summands. We call $(\U,\V)$ a \emph{cotorsion pair} if it satisfies the following conditions:
\begin{itemize}
\item[(a)] $\EE(\U,\V)=0$.
\smallskip

\item[(b)] For any object $B\in \B$, there exist two $\EE$-triangles
\begin{align*}
V_B\rightarrow U_B\rightarrow B{\dashrightarrow},\quad
B\rightarrow V^B\rightarrow U^B{\dashrightarrow}
\end{align*}
satisfying $U_B,U^B\in \U$ and $V_B,V^B\in \V$.
\end{itemize}
%A cotorsion pair is said to be hereditary if $\Omega \U\subseteq \U$.
\end{defn}

By the definition of a cotorsion pair, we can immediately conclude the following result.

\begin{lem}\label{y1}
Let $(\U,\V)$ be a cotorsion pair in $\B$.
\begin{itemize}
\item[(a)] $\V=\U^{\bot_1}:=\{ X\in \B \text{ }|\text{ } \EE(\U,X)=0\}$.
\item[(b)] $\U={^{\bot_1}}\V:=\{ Y\in \B \text{ }|\text{ } \EE(Y,\V)=0\}$.
\item[(c)] $\U$ and $\V$ are closed under extensions.
\item[(d)] $\mathcal I\subseteq \V$ and $\mathcal P\subseteq \U$.

\end{itemize}
\end{lem}

This lemma immediately yields the following conclusion.

\begin{cor}\label{cor1}
Let $(\U,\V)$ and $(\U',\V')$ be two cotorsion pair in $\B$. If $\U\subseteq \U'$ and $\V\subseteq \V'$, then $\U=\U'$ and $\V=\V'$.
\end{cor}

The following lemma gives a sufficient condition when a pair $(\U,\V)$ becomes a cotorsion pair.

\begin{lem}\label{2}
Assume that $\B$ has enough projectives. Let $\U$ and $\V$ be two subcategories of $\B$ which are closed under direct summands and $\EE(\U,\V)=0$.
If $\U$ is extension closed, $\mathcal P\subseteq\U$ and for any object $B\in \B$, there exists an $\EE$-triangle $$B\rightarrow V^B\rightarrow U^B\overset{}{\dashrightarrow}$$
where $U^B\in\U$ and $V^B\in\V$, then $(\U,\V)$ is a cotorsion pair in $\B$.
\end{lem}

\begin{proof}
By the definition of a cotorsion pair, we need to show that for any object $B\in \B$, there exists an $\EE$-triangle  $$V_B\rightarrow U_B\rightarrow B\overset{}{\dashrightarrow}$$
where $U_B\in\U$ and $V_B\in\V$.
Since $\B$ has enough projectives, there exists an $\EE$-triangle
$$\Omega B\rightarrow P\rightarrow B\overset{}{\dashrightarrow}$$
where $P\in\mathcal P\subseteq\U$. By hypothesis, there exists an $\EE$-triangle
$$\Omega B\rightarrow V\rightarrow U\overset{}{\dashrightarrow}$$
where $U\in\U$ and $V\in\V$.  By \cite[Proposition 3.15]{NP},
we have a commutative diagram made of $\EE$-triangles as follows.
$$
\xymatrix{\Omega B\ar[r]\ar[d]&P\ar[r]\ar[d]&B\ar@{-->}[r]\ar@{=}[d]&\\
V\ar[d]\ar[r]&U'\ar[d]\ar[r]&B\ar@{-->}[r]&\\
U\ar@{-->}[d]\ar@{=}[r]&U\ar@{-->}[d]&\\
&&}$$
Since $P,U\in\U$ and $\U$ is extension closed, we obtain $U'\in\U$.

Thus the $\EE$-triangle $V\rightarrow U'\rightarrow B\overset{}{\dashrightarrow}$
 is what we need.
\end{proof}

From now on, we assume that  $\B$ has enough projectives and enough injectives. Then according to \cite[Section 5]{LN}, we can define the higher extension functor $\EE^i$.

%\begin{defn}
%A subcategory $\W\subseteq \B$ is said to be $n$-rigid if $\EE^{i}(\W,\W)=0$ for any $1\leq i\leq n$. The subcategory $\G=\Cone(\W,\W)$ is said to be two-term $n$-rigid if it is extension closed and $\W$ is $n$-rigid.
%\end{defn}

%In this article, we always assume $\W$ is a $2$-rigid subcategory and $\G=\Cone(\W,\W)$ is the two-term $2$-rigid subcategory induced by $\W$. Note that any  two-term $2$-rigid subcategory is an extriangulated subcategory of $\B$.

The proof of the following lemma is left to the readers.

\begin{lem}\label{p0}
Let $\G$ be a proper $m$-term subcategory induced by $\W$. Then $\G$ is closed under direct summands.
\end{lem}

We show the following property for proper $2$-term subcategories.

\begin{lem}\label{p1}
Let $\G$ be a proper $m$-term subcategory induced by $\W$. Then $\G$ is closed under extensions, hence it is an extriangulated subcategory of $\B$. Moreover, $\mathcal G$ has enough porjectives $\W$. If $\mathcal I\subsetneq \W$, then $\G$ has enough injectives $\Sigma \W$, in particular, any injective object is projective-injective in $\mathcal G$.
\end{lem}

\begin{proof}
Assume we have an $\EE$-triangle $A\to B\to C\dashrightarrow$ where $A,C\in \G$, then we have the following commutative diagram
$$\xymatrix{
&W_1' \ar@{=}[r] \ar[d] &W_1' \ar[d] \\
A \ar[r] \ar@{=}[d] & A\oplus W_0' \ar[r] \ar[d] &W_0' \ar[d] \ar@{-->}[r]&\\
A \ar[r] &B \ar[r] \ar@{-->}[d] &C \ar@{-->}[r] \ar@{-->}[d] & \\
&&&
}
$$
where $W_0',W_1'\in \W$. The second row splits since $\W$ is $2$-rigid and $\G\subseteq \W^{\bot_1}$.  Then we have the following commutative diagram
$$\xymatrix{
W_1 \ar@{=}[r] \ar[d] &W_1 \ar[d]\\
W_1 \oplus W_1' \ar[r] \ar[d] &W_0\oplus W_0' \ar[r] \ar[d] &B \ar@{=}[d] \ar@{-->}[r] &\\
W_1' \ar[r] \ar@{-->}[d] &A\oplus W_0' \ar@{-->}[d] \ar[r] &B \ar@{-->}[r] &\\
&&
}
$$
which implies $B\in \G$.\\
Since $\G\subseteq \W^{\bot_1}$, $\W$ is the enough projectives in $\G$. \\
Assume $\mathcal I\subsetneq \W$. Any object $A\in \G$ admits a commutative diagram
$$\xymatrix{
W_1 \ar[r]^c \ar@{=}[d] &W_0 \ar[r]^a \ar[d]^i &A \ar@{-->}[r] \ar[d]^{a'} &\\
W_1 \ar[r] &I_0 \ar[r] \ar[d] &\Sigma W_1 \ar@{-->}[r] \ar[d] &\\
&\Sigma W_0 \ar@{=}[r] \ar@{-->}[d] &\Sigma W_0 \ar@{-->}[d]\\
&&
}
$$
where $W_1,W_0\in \W$ and $I_0\in \mathcal I$. For any object $\Sigma W$ where $W\in \W$, the morphisms in $\Hom_{\B}(W_1,\Sigma W)$ factor through $\mathcal I$, which implies it also factors through $c$, we have the following exact sequence
$$\Hom_{\B}(W_0,\Sigma W) \xrightarrow{\Hom_{\B}(c,\Sigma W)} \Hom_{\B}(W_1,\Sigma W)\xrightarrow{0} \EE(A,\Sigma W)\to \EE(W_0,\Sigma W)=0.$$
This implies that $\Sigma \W$ is the enough injectives in $\G$. Moreover, since $\mathcal I\subseteq \W\cap \Sigma \W$, we get that $\mathcal I$ is projective-injective in $\mathcal G$.
\end{proof}

\section{Cotorsion pairs and support $\tau$-titling subcategories}

In this section, let $\G=\Cone(\W,\W)$ be a proper $2$-term subcategory. From this section, we assume that $\B$ satisfies condition (WIC), see \cite[Condition 5.8]{NP}:

\begin{itemize}
\item If we have a deflation $h: A\xrightarrow{~f~} B\xrightarrow{~g~} C$, then $g$ is also a deflation.
\item If we have an inflation $h: A\xrightarrow{~f~} B\xrightarrow{~g~} C$, then $f$ is also an inflation.
\end{itemize}

Note that this condition automatically holds on triangulated categories and Krull-Schmidt exact categories.

We show the following extriangulated version Wakamatsu's Lemma.

\begin{lem}\label{Wa}
Let $\D$ be an extension closed subcategory of $\B$.  If we have an $\EE$-triangle $A\xrightarrow{d} D \to B\dashrightarrow$ where $d$ is a minimal
 left $\D$-approximation, then $B\in {^{\bot_1}}\D$.
\end{lem}

\proof Applying the functor $\Hom_{\B}(-,D')$ with $D'\in\D$ to the $\EE$-triangle $A\xrightarrow{d} D \to B\dashrightarrow$, we have the following
exact sequence:
$$\Hom_{\B}(D,D')\xrightarrow{\Hom_{\B}(d,D')}\Hom_{\B}(A,D')\xrightarrow{\alpha}\EE(B,D')\xrightarrow{\beta} \EE(D,D')\xrightarrow{\EE(d,D')}\EE(A,D').$$
Since $d$ is a left $\D$-approximation, we get that $\Hom_{\B}(d,D')$ is an epimorphism.
It follows that $\alpha=0$. We claim that the $\EE(d,D')$ is a monomorphism whence $\beta=0$. This forces $\EE(B,D')=0$ which implies $B\in {^{\bot_1}}\D$.

To see that the $\EE(d,D')$ is a monomorphism, let $\eta\in\EE(D,D')$ be any $\EE$-extension, realized by
an $\EE$-triangle $$\xymatrix{D'\ar[r]^u&C\ar[r]^{v}&D\ar@{-->}[r]^{\eta}&}$$ such that
$\EE(d,D')(\eta)=0$. We obtain a morphism of $\EE$-triangles
$$\xymatrix{D'\ar[r]^{x}\ar@{=}[d]&M\ar[r]^y\ar[d]^g&A\ar@{-->}[r]^{0}\ar[d]^d&\\
D'\ar[r]^{u}&C\ar[r]^{v}&D\ar@{-->}[r]^{\eta}&.}$$
By \cite[Corollary 3.5]{NP}, we know that $y$ is a retraction, hence there exists a morphism $y'\colon A\to M$ such that $yy'=1$.
Since $\D$ is extension closed, we have $C\in\D$. Since $d$ is a left $\D$-approximation, there exists a morphism $w\colon D\to C$ such that $wd=gy'$.
It follows that $d=dyy'=vgy'=vwd$. Since $d$ is left minimal, we have that $vw$ is an isomorphism.
In particular, $v$ is a retraction. By \cite[Corollary 3.5]{NP}, $\eta$ splits and then $\eta=0$,  as desired.  \qed

\begin{rem}
Under the condition {\rm (WIC)}, if we have an inflation $d:A\to D$ which is a left $\D$-approximation, then we can take a minimal left $\D$-apporximation $d_0:A\to D_0$ which is still an  inflation.
\end{rem}

We have the following corollary, which is a special case of the dual of Lemma \ref{2}.

\begin{cor}\label{c1}
Let $\D$ be an extension closed subcategory which contains $\mathcal I$. If any object $X$ admits an $\EE$-triangle $X\to D\to E\dashrightarrow$ where $D\in \D$ and $E\in {^{\bot_1}}\D$, then $({^{\bot_1}}\D,\D)$ is a cotorsion pair.
\end{cor}

We introduce the following notions.

Let $\B'$ be a subcategory of $\B$. Denote by $[\B'](A,B)$ the subgroup of $\Hom_{\B}(A,B)$ such that $f\in [\B'](A,B)$ if $f$ factors through an object in $\B'$.

We denote by $\underline \B'$ the category which has the same objects as $\B'$, and $$\Hom_{\uB'}(A,B)=\Hom_{\B}(A,B)/[\mathcal I](A,B)$$ where $A,B\in \B'$. For any morphism $f\in \Hom_{\B}(A,B)$, we denote its image in $\Hom_{\uB'}(A,B)$ by $\underline f$.

We denote by $\overline \B'$ the category which has the same objects as $\B'$, and $$\Hom_{\oB'}(A,B)=\Hom_{\B}(A,B)/[\Sigma \W](A,B)$$ where $A,B\in \B'$. For any morphism $f\in \Hom_{\B}(A,B)$, we denote its image in $\Hom_{\oB'}(A,B)$ by $\overline f$.

\smallskip

%From now on, we assume that $\mathcal I\subsetneq \W$ and $\overline {\mathcal G}$ is abelian.
In this section, let $\W$ be a $2$-rigid subcategory and $\G$ be a $2$-term subcategory induced by $\W$. We assume that $\mathcal I\subsetneq \W$.

%\begin{rem}
%The conditions above are satisfied when $\W$ is a $d$-cluster tilting subcategory, $d\geq 3$.
%\end{rem}

\begin{lem}\label{exact}
If we have $\EE$-triangle $X\xrightarrow{~x~} Y\xrightarrow{~y~} Z\dashrightarrow$ in $\mathcal G$, then
$$\Hom_{\oB}(W,X)\xrightarrow{\Hom_{\oB}(W,\overline x)} \Hom_{\oB}(W,Y)\xrightarrow{\Hom_{\oB}(W,\overline y)} \Hom_{\oB}(W,Z)\to 0$$
is exact for any $W\in \overline \W$.
\end{lem}

\begin{proof}
Since $\W$ is projective in $\G$, then for $\EE$-triangle $X\xrightarrow{~x~} Y\xrightarrow{~y~} Z\dashrightarrow$ in $\mathcal G$, we get an exact sequence
$$\Hom_{\B}(W,X)\xrightarrow{\Hom_{\B}(W, x)} \Hom_{\B}(W,Y)\xrightarrow{\Hom_{\B}(W,y)} \Hom_{\B}(W,Z)\to 0.$$
Hence $\Hom_{\oB}(W,\overline y)$ is an epimorphism and $\Hom_{\oB}(W,\overline {yx})=0$. For any morphism $\overline w\in  \Hom_{\oB}(W,Y)$ such that $\overline {yw}=0$, since $\Hom_{\uB}(\W,\Sigma \W)=0$, we have $yw: W\xrightarrow{w_1} I \xrightarrow{w_2} Z$ where $I\in \mathcal I$. Since $I$ is projective in $\G$, there is a morphism $w_3:I\to Y$ such that $yw_3=w_2$, hence $y(w-w_3w_1)=0$. Then there is a morphism $w':W\to X$ such that $w_1w'=w-w_3w_1$. Hence $\overline w=\overline {w_1w'}$, which implies that
$$\Hom_{\oB}(W,X)\xrightarrow{\Hom_{\oB}(W,\overline x)} \Hom_{\oB}(W,Y)\xrightarrow{\Hom_{\oB}(W,\overline y)} \Hom_{\oB}(W,Z)\to 0$$
is an exact sequence.
\end{proof}

\begin{lem}
$\overline \G$ has enough projectives $\overline \W$.
\end{lem}

\begin{proof}
For any object $A\in \G$, we have the following commutative diagram
$$\xymatrix{
W_1 \ar[r] \ar@{=}[d] &W_0 \ar[r]^a \ar[d]^i &A \ar@{-->}[r] \ar[d]^{a'} &\\
W_1 \ar[r] &I_0 \ar[r] \ar[d] &\Sigma W_1 \ar@{-->}[r] \ar[d] &\\
&\Sigma W_0 \ar@{=}[r] \ar@{-->}[d] &\Sigma W_0 \ar@{-->}[d]\\
&&
}
$$
where $W_1,W_0\in \W$ and $I_0\in \mathcal I$. For any morphism $x:A\to X$ such that $\overline {xa}=0$, we have $xa$ factors through $\mathcal I$, hence it factors through $i$. Then there is a morphism $w_1:\Sigma W_1\to X$ such that $w_1a'=x$, which implies that $\overline x=0$. Hence $\overline a$ is an epimorphism. \\
For an epimorphism $\overline f:A\to B$ in $\overline \G$, we have the following commutative diagram.
$$\xymatrix{
W_0 \ar[r] \ar[d] & I_0 \ar[r] \ar[d] &\Sigma W_0 \ar@{=}[d] \ar@{-->}[r] &\\
A \ar[r]^{a'} \ar[d]_f &\Sigma W_1 \ar[r] \ar[d]^h &\Sigma W_0 \ar@{=}[d] \ar@{-->}[r] &\\
B \ar[r]^g &C \ar[r] &\Sigma W_0\ar@{-->}[r] &
}
$$
Since $\G$ is extension closed, we have $C\in \G$.
%which induces an $\EE$-triangle $W_0 \to B\oplus I_0 \to C\dashrightarrow$. Since $B$ admits an $\EE$-triangle $W_1'\to W_0'\to B \dashrightarrow$ where $W_1',W_0'\in \W$, we have the following commutative diagram
%$$\xymatrix{
%W_1' \ar@{=}[r] \ar[d] &W_1' \ar[d]\\
%W_1'\oplus W_0 \ar[r] \ar[d] &W_0'\oplus I_0 \ar[r] \ar[d] &C \ar@{=}[d] \ar@{-->}[r]  &\\
%W_0 \ar[r] \ar@{-->}[d] &B\oplus I_0 \ar[r] \ar@{-->}[d] &C \ar@{-->}[r] &\\
%&&
%}
%$$
%which implies $C\in \G$.
Since $\overline f$ is an epimorphism and $\overline {gf}=0$, we get $g$ factors through $\Sigma \W$. By Lemma \ref{exact}, for any object $W$, we have the following exact sequence
$$\Hom_{\oB}(W, B) \xrightarrow{\Hom_{\oB}(W,\overline g)=0} \Hom_{\oB}(W,C)\to  \Hom_{\oB}(W,\Sigma W_0)=0$$
which implies $ \Hom_{\oB}(W,C)=0$. $C$ admits the following commutative diagram
$$\xymatrix{
W_1'' \ar[r] \ar@{=}[d] &W_0''\ar[r]^c \ar[d]^{i'} &C\ar@{-->}[r] \ar[d] &\\
W_1'' \ar[r] &I_0'' \ar[r] \ar[d] &\Sigma W_1'' \ar@{-->}[r] \ar[d]^{w'} &\\
&\Sigma W_0'' \ar@{=}[r] \ar@{-->}[d] &\Sigma W_0'' \ar@{-->}[d]\\
&&
}
$$
where $W_1'',W_0''\in \W$ and $I_0''\in \mathcal I$. Since $c$ factors through $\mathcal I$, it factors through $i'$, which implies that $1_{\Sigma W_0''}$ factors through $w'$. Hence $C$ is a direct summand of $\Sigma W_1''$, which means $C\in \Sigma \W$. Since we have an $\EE$-triangle $A\xrightarrow{\svecv{f}{a'}} B\oplus \Sigma W_1 \xrightarrow{\svech{g}{-h}} C\dasharrow$, by Lemma \ref{exact}, we get an epimorphism $\Hom_{\oB}(W, A) \xrightarrow{\Hom_{\oB}(W,\overline f)=0} \Hom_{\oB}(W,B)\to 0$ for any object $W\in \overline \W$. Hence $\overline \W$ is the enough projetives in $\overline \G$.
\end{proof}

We have the following proposition.

\begin{prop}\label{prop}
$\overline {\mathcal G} \simeq \mod \overline \W\simeq \mod \underline \W$.
\end{prop}

\begin{proof}
Since $\overline \G$ has enough projectives $\overline \W$,
\begin{eqnarray*}
\mathbb{F}\colon \overline \G & \longrightarrow & \mod \overline \W\\
 M& \longmapsto & \Hom_{\oB}(-,M)|_{\overline \W}
\end{eqnarray*}
is an equivalence.

Since $\overline \W=\underline \W$, we have $\mod \overline \W\simeq \mod \underline \W$.
\end{proof}

By Proposition \ref{prop} and Lemma \ref{exact}, we have the following corollary.

\begin{cor}\label{p2}
If we have an $\EE$-triangle $X\xrightarrow{~x~} Y\xrightarrow{~y~} Z\dashrightarrow$ in $\mathcal G$, then $X\xrightarrow{~\overline x~} Y\xrightarrow{~\overline y~} Z\to 0$ is exact in $\overline {\mathcal G}$.
\end{cor}

The following lemma is useful.

\begin{lem}\label{p3}
For any indecomposable objects $X,Y\in \G$, $\EE(X,Y)=0$ implies $\Ext^1_{\overline \G}(X,Y)=0$.
\end{lem}

\begin{proof}
Let
$$(\blacklozenge) \quad 0\to Y\xrightarrow{g} Z\xrightarrow{f} X\to 0 \text{ }$$
 be a short exact sequence in $\overline \G$. The morphism $g$ admits the following commutative diagram
$$\xymatrix{
Y \ar[r]^w \ar[d]_g &\Sigma W_0 \ar[r] \ar[d]^{v} &\Sigma W_1 \ar@{=}[d] \ar@{-->}[r] &\\
Z \ar[r]_{f'} &X' \ar[r] &\Sigma W_1\ar@{-->}[r] &
}
$$
where $W_0,W_1 \in \W$. Then we have an $\EE$-triangle $Y\xrightarrow{\svecv{w}{g}} \Sigma W_0\oplus Z\xrightarrow{\svech{-v}{f'}} X'\dashrightarrow$, which induces an exact sequence $(\lozenge) \quad 0 \to Y\xrightarrow{\overline g} Z\xrightarrow{\overline f'} X'\to 0$. It is isomorphic to $(\blacklozenge)$.  Hence $X'=X\oplus X_0$ where $X_0\in \Sigma\W$. Denote morphism $\Sigma W_0\oplus Z \xrightarrow{\svech{-v}{f'}} X'$ by $\Sigma W_0\oplus Z \xrightarrow{\left(\begin{smallmatrix}
-v_1& f_1\\
-v_2& f_2
\end{smallmatrix}\right)} X\oplus X_0$, then $\overline f_1=\overline f'$. Since $\EE(X,Y)=0$, we have a commutative diagram.
$$\xymatrix{
&&&X\ar[d]^-{\svecv{1}{0}} \ar@{.>}[dll]_-{\svecv{x_1}{x_2}} \\
Y\ar[r]_-{\svecv{w}{g}} &\Sigma W_0\oplus Z\ar[rr]_-{\left(\begin{smallmatrix}
-v_1& f_1\\
-v_2& f_2
\end{smallmatrix}\right)} &&X\oplus X_0 \ar@{-->}[r] &
}
$$
Then we have $\overline {f_1x_2}=\overline 1$ Hence  $(\lozenge)$ splits, so is $(\blacklozenge)$.
\end{proof}

For a subcategory $\overline \N\subseteq \overline{\mathcal G}$, denote  $\rm\bf P(\overline \N)$ by the subcategory $\{N\in  \overline \N \text{ }|\text{ } \Ext^1_{\overline{\mathcal G}}(N,\overline \N)=0 \};$ denote by $\Fac \overline \N$ by the subcategory $\{B\in  \overline \G \text{ }|\text{ there is an epimorphism } N\to B\to 0,\text{ where } N\in \N \}.$

\begin{lem}\label{p4}
Let $(\X,\Y)$ be a cotorsion pair in $\G$. Then $\overline \Y=\Fac \overline \Y=\Fac (\overline \X\cap\overline \Y)$.
\end{lem}

\begin{proof}
We first show that $\overline \Y\supseteq \Fac \overline \Y$.\\
Let $Z\in  \Fac \overline \Y$. Then $Z$ admits an epimorphism $\overline y\colon Y\to Z\to 0$ where $Y\in \Y$. Thus we can get the following commutative diagram
$$\xymatrix{
Y \ar[r] \ar[d]_y &\Sigma W^1 \ar[d] \ar[r] &\Sigma W^2 \ar@{-->}[r] \ar@{=}[d] &\\
Z \ar[r]^z &A \ar[r] &\Sigma W^2 \ar@{-->}[r] &
}
$$
where $W^1,W^2\in \W$ and $A\in \G$. It induces an $\EE$-triangle $Y\to Z\oplus \Sigma W^1\to  A\dashrightarrow$. By Corollary \ref{p2}, we get an exact sequence $Y\xrightarrow{\overline y} Z\xrightarrow{\overline z} A\to 0$ in $\overline \G$. Hence $\overline z=0$ and $A\in \Sigma \W$. Then we have $Z\in \Y$.
Now we show $\overline \Y\subseteq \Fac (\overline \X\cap\overline \Y)$.\\
Any object $Y\in \Y$ admits an epimorphism $W_0\xrightarrow{~\overline y~} Y\to 0$.
Since $W_0\in \mathcal G$ and $(\X,\Y)$ is a cotorsion pair in $\mathcal G$, there exists an $\EE$-triangle
 $$\xymatrix{
W_0\ar[r]^{y_0} &Y_0\ar[r] &X_0\ar@{-->}[r]&}
$$
where $Y_0\in \Y$ and $X_0\in \X$. Since $W_0\in \X$, we have $Y_0\in \X$. We have the following commutative diagram of $\EE$-triangles
$$\xymatrix{
W_0\ar[r]^{y_0} \ar@{=}[d] &Y_0\ar[r] \ar[d] &X_0\ar@{-->}[r] \ar[d] &\\
W_0 \ar[r] &I_0 \ar[r] &\Sigma W_0 \ar@{-->}[r] &
}
$$
where $I_0\in \mathcal I$. It induces an exact sequence $W_0\xrightarrow{\overline {y_0}} Y_0\xrightarrow{\overline {x_0}} X_0\to 0$. It also induces an $\EE$-triangle $Y_0\to X_0\oplus I_0 \to \Sigma W_0\dashrightarrow$
 which implies $X_0\in \Y$. Since $y_0$ is a left $\Y$-approximation of $W_0$, there is a morphism $a:Y_0\to Y$ such that $y=ay_0$. Moreover, $\overline a $ is an epimorphism. Hence $\overline \Y\subseteq  \Fac (\overline \X\cap\overline \Y)$. This shows that $\overline \Y=\Fac \overline \Y=\Fac (\overline \X\cap\overline \Y)$.
\end{proof}

According to this lemma, we have the following proposition.

\begin{prop}
Let $(\X,\Y)$ be a cotorsion pair on $\G$ and $Z$ be an object in $\G$. If $\EE(\X\cap\Y,Z)=0$ and $\EE(Z,\X\cap\Y)=0$, then $Z\in \X\cap \Y$. This means $\X\cap \Y$ is maximal rigid in $\G$.
\end{prop}

\begin{proof}
We can assume that $Z$ is indecomposable. Since $Z\in \G$, it admits an $\EE$-triangle $W_1\to W_0 \xrightarrow{w_0} Z\dashrightarrow$ where $W_1, W_0\in \W$. By the proof of Lemma \ref{p4}, $W_0$ admits an $\EE$-triangle $W_0\xrightarrow{y_0} Y_0\to X_0 \dashrightarrow$ where $Y_0,X_0\in \X\cap \Y$. Since $\EE(X_0,Z)=0$, there is a morphism $z:Y_0\to Z$ such that $zy_0=w_0$. Since $\overline w_0$ is an epimorphism in $\G$, $\overline z$ is also an epimorphism, which implies $Z\in \Fac(\overline \X\cap \overline \Y)=\overline \Y$ by Lemma \ref{p4}. Then $Z$ admits an $\EE$-triangle $Y_1\to X_1\to Z\dashrightarrow$ where $X_1\in \X\cap\Y$ and $Y_1\in \Y$. Since $\G=\Cone(\W,\W)$ has enough projectives $\W$ and enough injectives, the higher extensions will vanish in $\G$, hence we have an exact sequence $0=\EE(X_1,\Y)\to \EE(Y_1,\Y)\to 0$ which implies $Y_1\in \X\cap\Y$. Then we have $\EE(Z,Y)=0$, $Z$ becomes a direct summand of $X_1$. Hence $Z\in \X\cap\Y$.
\end{proof}

Based on \cite[Definition 1.3]{IJY}, we have the following definition.

\begin{defn}\label{tau}
Let $\overline \M$ be a subcategory in $\overline {\mathcal G}$.

\begin{itemize}
\item[\rm (i)] $\overline \M$ is said to be $\tau$-rigid if any object $M\in \overline \M$ admits an exact sequence $W_1\xrightarrow{\overline f} W_0\to M\to 0$ such that $ W_1,W_0\in \overline { \W}$ and $\Hom_{\oB}(\overline f, M')$ is a surjection for any $M'\in \overline \M$.

\item[\rm (ii)] $\overline \M$ is said to be support $\tau$-tilting if it is $\tau$-rigid and any projective object $W$ admits an exact sequence $W\xrightarrow{\overline m} M^0\to M^1\to 0$ such that $M^0,M^1\in \overline \M$ and $\overline m$ is a left $\overline \M$-approximation.

\end{itemize}
\end{defn}

By the proof of \cite[Lemma 5.2 and Proposition 5.3]{IJY}, we get the following important lemmas.

\begin{lem}\label{p5}
A subcategory $\overline \M\subseteq \overline \G$ is $\tau$-rigid if and only if $\Ext^1_{\overline \G}(\overline \M, \Fac \overline \M)=0$.
\end{lem}

\begin{lem}\label{p6}
If $\overline \M\subseteq \overline \G$ is $\tau$-rigid, then $\Fac \overline \M$ is closed under extensions.
\end{lem}

We have the following corollary.

\begin{cor}\label{impor}
Let $(\X,\Y)$ be a cotorsion pair in $\G$. Then $\overline \X\cap \overline \Y$ is support $\tau$-tilting. Moreover, $\overline \X\cap \overline \Y=\rm\bf P(\overline \Y)$.
\end{cor}

\begin{proof}
By Lemma \ref{p4} and \ref{p3}, we have $\Ext^1_{\overline \G}(\overline \X\cap \overline \Y, \Fac (\overline \X\cap \overline \Y))=\Ext^1_{\overline \G}(\overline \X\cap \overline \Y, \overline \Y)=0$. Hence $\overline \X\cap \overline \Y$ is $\tau$-rigid. According to the proof of Lemma \ref{impor}. any projective object $W_0$ in $\G$ admits an exact sequence $W_0\xrightarrow{\overline {y_0}} Y_0\xrightarrow{\overline {x_0}} X_0\to 0$ where $X_0,Y_0\in \X\cap \Y$ and $y_0$ is a left $\Y$-approximation. By Definition \ref{tau}, it is support $\tau$-tilting. We also have $\overline \X\cap \overline \Y=\rm\bf P(\Fac (\overline \X\cap \overline \Y))=\rm\bf P(\overline \Y).$
\end{proof}

Now we can state and show the main theorem of this section.

\begin{thm}\label{main2}
Let $\W$ be a $2$-rigid subcategory such that $\mathcal I\subsetneq \W$ and $\mathcal G=\Cone(\W,\W)$ be the proper $2$-term subcategory induced by $\W$. Assume that $\overline{\mathcal G}$ is abelian. Then we have a one-to-one correspondence $$\Phi\colon (\X,\Y) \mapsto \overline \X\cap\overline \Y$$ from the first of the following sets to the second.
\begin{itemize}
\item[\rm (i)] Cotorsion pairs $(\X,\Y)$ in $\mathcal G$.
\smallskip

\item[\rm (ii)] Support $\tau$-tilting subcategories $\overline \M$ in $\overline {\mathcal G}$.
\end{itemize}
\end{thm}

\begin{proof}

By Corollary \ref{impor}, we have that $\Phi$ is injective.
\smallskip

Now we show that $\Phi$ is surjective. Let $\overline \M$ be a support $\tau$-tilting subcategory in $\overline \G$. Let $\overline \N=\Fac \overline \M$. We assume $\N\cap \Sigma \W=0$, let $\Y=\add(\N\cup \Sigma \W)$. We show $(\X={^{\bot_1}}\Y\cap \G,\Y)$ is a cotorsion pair in $\G$. Then $ \overline \X\cap\overline \Y=\mathbf{P} (\overline \Y)=\mathbf{P} (\Fac \overline \M)=\overline \M$. %We will apply Corollary \ref{c1}.
\smallskip

%(i) $\Y$ is extension closed.
%\smallskip

Let $Y_0\xrightarrow{ y_0} Y_1\xrightarrow{ y_1} Y_2\dashrightarrow$ be an $\EE$-triangle where $Y_0,Y_2\in \Y$. By Lemma \ref{p2}, we have an exact sequence $Y_0\xrightarrow{\overline {y_0}} Y_1\xrightarrow{\overline  {y_1}} Y_2\to 0$ in $\overline \G$. It induces a short exact sequence $0\to \Im \overline {y_0} \to Y_1\to Y_2\to 0$. Since $\Im \overline {y_0}, Y_2\in \overline \N$, by Lemma \ref{p6}, we have $Y_1\in  \overline \N$. Hence $Y_1\in \Y$, $\Y$ is extension closed.
\smallskip

%(ii) $\Y$ is covariantly finite.
%\smallskip

Any object $Z\in \G$ admits an $\EE$-triangle $W_0\to Z\oplus I_1 \to \Sigma W_1\dashrightarrow$ where $W_0,W_1\in \W$, $I_1\in \mathcal I$. Since $\overline \M$ is support $\tau$-tilting, we have left $\overline \M$-approximation $\overline m_1:W_0\to M_1$. Then we can get the following commutative diagram
$$\xymatrix{
W_0 \ar[r]^{i_1} \ar[d]_{m_1} &I_0 \ar[r] \ar[d] &\Sigma W_0 \ar@{=}[d] \ar@{-->}[r] &\\
M_1 \ar[r] &X_1 \ar[r] &\Sigma W_0\ar@{-->}[r] &
}
$$
where $I_0\in \mathcal I$. Since $M_1,\Sigma W_0\in \Y$, we have $X_1\in \Y\subseteq \G$. It is not hard to check that in the $\EE$-triangle
$$W_0 \xrightarrow{\svecv{m_1}{i_1}=\alpha} M_1\oplus I_0 \rightarrow X_1\dashrightarrow,$$ $\alpha$ is a left $\Y$-approximation. Hence $W_0$ admits an $\EE$-triangle
$W_0 \xrightarrow{w_0} Y_0\to X_0\dashrightarrow$ where $w_0$ is a minimal left $\Y$-approximation. Then by Lemma \ref{Wa}, $X_0\in {^{\bot_1}}\Y$. Moreover, $X_0$ is a direct summand of $X_1$, by Lemma \ref{p0}, $X_0\in \G$. Hence $X_0\in \X$. Now we have the following commutative diagram
$$\xymatrix{
W_0 \ar[r]^-{\alpha} \ar[d]_{w_0} &Z_1\oplus I_0 \ar[r] \ar[d]^-{\svech{a}{b}} &\Sigma W_1 \ar@{=}[d] \ar@{-->}[r] &\\
Y_0 \ar[r] \ar[d] &Y_1 \ar[r] \ar[d] &\Sigma W_1  \ar@{-->}[r] &\\
X_0 \ar@{-->}[d] \ar@{=}[r] &X_0 \ar@{-->}[d]\\
&&}
$$
where $Y_1\in \Y$. Since $\svech{a}{b}$ is a left $\Y$-approximation, $Z_1$ admits an $\EE$-triangle $Z_1\xrightarrow{a} Y_1\to X_0'\dashrightarrow$ where $a$ is a left $\Y$-approximation. Hence we can get an $\EE$-triangle $Z_1\xrightarrow{a'} Y_1'\to X_0''\dashrightarrow$ where $a'$ is a minimal left $\Y$-approximation. Then $X_0''$ becomes a direct summand of $X_0$, hence $X_0''\in \X$.
By Corollary \ref{c1}, we obtain that $(\X,\Y)$ is a cotorsion pair in $\G$.
\end{proof}

%\begin{prop}
%Let $\A$ be a notherian abelian category with enough projectives. Then we have a one-to-one correspondence from the first of the following sets to the second.
%\begin{itemize}
%\item[\rm (i)] Functorially finite torsion classes $\T$;
%\smallskip

%\item[\rm (ii)] Support $\tau$-tilting subcategories $\M$.
%\end{itemize}
%\end{prop}

%We have the following proposition when $\B$ becomes triangulated.

%\begin{prop}
%Let $\B$ be a triangulated category. Then we have one-to-one correspondences between the following sets.
%\begin{itemize}
%\item[\rm (i)] Maximal rigid objects in $\G$.
%\smallskip

%\item[\rm (ii)] Cotorsion pairs $(\X, \Y)$ in $\G$.
%\smallskip

%\item[\rm (iii)] Support $\tau$-tilting subcategories $\overline \M$ in $\overline \G$.
%\end{itemize}
%\end{prop}

%\begin{proof}
%Since $\G=\W*\Sigma \W$, by \cite[Theorem 4.5]{ZZ3}, there is a one-to-one correspondence between support $\tau$-tilting subcategories in $\overline \G$ and maximal relative rigid subcategories in $\G$. By the result of \cite{FGL}, since $\W$ is $2$-rigid, any relative rigid subcategory is rigid, we have a one-to-one correspondence between maximal rigid objects in $\G$ and support $\tau$-tilting subcategories $\overline \G$.
%\end{proof}

\section{Hereditary cotorsion pairs}

\begin{defn}\label{def2}
Let $\B'$ and $\B''$ be two subcategories in $\B$.
\begin{itemize}
\item[(a)] Denote by $\CoCone(\B',\B'')$ the subcategory
$$\{X\in \B \text{ }|\text{ } \textrm{there exists an}~ \text{ } \EE\text{-triangle } \xymatrix@C=0.5cm@R0.5cm{ X \ar[r] &B' \ar[r] &B'' \ar@{-->}[r] &} \text{, }B'\in \B' \text{ and }B''\in \B'' \}.$$

\item[(b)] Let $\Omega \B'=\CoCone(\mathcal P,\B')$. We write an object $D$ in the form $\Omega B$ if it admits an $\EE$-triangle $\xymatrix@C=0.5cm@R0.5cm{D \ar[r] &P \ar[r] &B \ar@{-->}[r] &}$ where $P\in \mathcal P$.
\smallskip

\end{itemize}
\end{defn}

\begin{defn}\label{hered}
A cotorsion pair is called hereditary if $\Omega \U\subseteq \U$.
\end{defn}

The following lemma gives several equivalent conditions of a cotorsion pair being hereditary.

\begin{lem}\label{y2}
Let $(\U,\V)$ be a cotorsion pair in $\B$. Then the following statements are equivalent:
\begin{itemize}
\item[\rm (I)] $\Omega \U\subseteq \U$;
\item[\rm (II)] $\Sigma \V\subseteq \V$;
\item[\rm (III)] $\EE^2(\U,\V)=0$;
\item[\rm (IV)] $\EE^i(\U,\V)=0$, $i\geq 1$.
\end{itemize}
\end{lem}

\begin{proof}
For any $V\in\V$, there exists an $\EE$-triangle
$$V\rightarrow I\rightarrow \Sigma V\overset{}{\dashrightarrow}$$
where $I\in\mathcal I$. Applying the functor ${\rm Hom}_\B(\U,-)$ to the above $\EE$-triangle,
we have the following exact sequence:
$$0=\EE^i(\U,I)\to\EE^i(\U,\Sigma V)\to \EE^{i+1}(\U,V)\to \EE^{i+1}(\U,I)=0, i\geq 1.$$
Thus $\EE^i(\U,\Sigma V)\simeq \EE^{i+1}(\U,V)$. If $\EE^2(\U,\V)=0$, then $\EE(\U,\Sigma V)=0$, which implies $\Sigma \V\subseteq \V$. On the other hand, if  $\Sigma \V\subseteq \V$, we have $\EE^2(\U,\V)=0$. Hence (II) and (III) are equivalent.
\smallskip

By the same method, we can show that $\Omega \U\subseteq \U$ if and only if $\EE^2(\U,\V)=0$, then (I) and (III) are equivalent.
\smallskip

(III) $\Rightarrow$ (IV). When $i=2$, we have $\EE^2(\U,\Sigma \V)=0$ since $\Sigma \V\subseteq \V$. Then we get $\EE^3(\U,\V)=0$. Proceeding inductively, we deduce that
$\EE^i(\U, \V)=0$ for any $i\geq 1$.

\smallskip

(IV) $\Rightarrow$ (III).  This is trivial.
\end{proof}

In this section, let $(\U,\V)$ be a hereditary cotorsion pair in $\B$. A cotorsion pair $(\U',\V')$ is called \emph{intermediate} if $\U\subseteq \U'\subseteq {^{\bot_1}}(\Sigma \V)$. Denote $\U\cap\V$ by $\W$, we call $\W$ the \emph{co-heart} of $(\U,\V)$.

\begin{defn}\label{def2}
For convenience, we give the following notations.
\begin{itemize}
\item[(a)] Denote $\Cone(\Sigma \V,\mathcal I)$ by $\Sigma^2 \V$ and $\Cone(\Sigma^i \V,\mathcal I)$ by $\Sigma^{i+1} \V$.
\item[(b)] Denote $\Cone(\W,\W)$ by $\W_1$ and  $\Cone(\W_i,\W)$ by $\W_{i+1}$. We also denote $\W$ by $\W_0$
\end{itemize}
\end{defn}

%\begin{rem}
%We can call  $\W_i$ an $(i+1)$-term subcategory.
%\end{rem}

We show the following lemma.

\begin{lem}\label{imp}
$\V\cap {^{\bot_1}}(\Sigma^{i} \V)=\W_{i}$, $i> 0$.
\end{lem}

\proof We first show the case when $i=1$. \\
For any $M\in\W_1$, there exists an $\EE$-triangle
\begin{equation}\label{t2}
W_2\rightarrow W_1\rightarrow M\overset{}{\dashrightarrow}
\end{equation}
where $W_1,W_2\in\W$. Applying the functor ${\rm Hom}_\B(\U,-)$ to the $\EE$-triangle (\ref{t2}),
we have the following exact sequence:
$$0=\EE(\U,W_1)\to\EE(\U, M)\to \EE^2(\U,W_2)=0.$$
It follows that $\EE(\U,M)=0$, then $M\in\V$ by Lemma \ref{y1}.

Applying the functor ${\rm Hom}_\B(-,\V)$ to the $\EE$-triangle (\ref{t2}),
we have the following exact sequence:
$$0=\EE(W_2,\V)\to\EE^2(M,\V)\to \EE^2(W_1,\V)=0.$$
It follows that $\EE^2(M,\V)=0$.
Note that $\EE(M,\Sigma\V)\simeq \EE^2(M,\V)=0$, thus we have $M\in{^{\bot_1}}(\Sigma \V)$.

Conversely, for any $N\in\V\cap {^{\bot_1}}(\Sigma \V)$, since $(\U,\V)$ is a cotorsion pair,
there exists an $\EE$-triangle
\begin{equation}\label{t3}
V_0\rightarrow U_0\rightarrow N\overset{}{\dashrightarrow}
\end{equation}
where $U_0\in\U$ and $V_0\in\V$.

Since $V_0, N\in\V$ and $\V$ is extension closed, we obtain $U_0\in\V$ which implies $U_0\in\U\cap\V=\W$.

Applying the functor ${\rm Hom}_\B(-,\V)$ to the $\EE$-triangle (\ref{t3}),
we have the following exact sequence:
$$0=\EE(U_0,\V)\to\EE(V_0, \V)\to \EE^2(N,\V)=0.$$
It follows that $\EE(V_0,\V)=0$ implies $V_0\in\U$ by Lemma \ref{y1}.
That is to say, $V_0\in\U\cap\V=\W$. Hence $N\in\W_1$.

Now we assume the lemma holds when $i=k-1$. Let $i=k$. For any $M\in\Cone(\W_i,\W)$, there exists an $\EE$-triangle
\begin{equation}\label{ta}
Q\rightarrow W\rightarrow M\overset{}{\dashrightarrow}
\end{equation}
where $W\in\W$ and $Q\in \W_i$. Then by hypothesis, $Q\in \V\cap {^{\bot_1}}(\Sigma^i \V)$. Applying the functor ${\rm Hom}_\B(\U,-)$ to this $\EE$-triangle (\ref{ta}),
we have the following exact sequence:
$$0=\EE(\U,W)\to\EE(\U, M)\to \EE^2(\U,Q)\simeq \EE^{i+2}(U,Q)=0.$$
It follows that $\EE(\U,M)=0$, then $M\in\V$ by Lemma \ref{y1}.

Applying the functor ${\rm Hom}_\B(-,\Sigma^k \V)$ to the $\EE$-triangle (\ref{t2}),
we have the following exact sequence:
$$0=\EE(Q,\Sigma^k \V)\to\EE^2(M,\Sigma^k \V)\to \EE^2(W_1,\Sigma^k \V)=0.$$
It follows that $\EE(M,\Sigma^{k+1} \V)\simeq \EE^2(M,\Sigma^k \V)=0$, thus we have $M\in{^{\bot_1}}(\Sigma^{k+1} \V)$.

On the other hand, for any $N\in\V\cap {^{\bot_1}}(\Sigma^{k+1} \V)$, there exists an $\EE$-triangle
\begin{equation}\label{t3}
V_0\rightarrow U_0\rightarrow N\overset{}{\dashrightarrow}
\end{equation}
where $U_0\in\W$ and $V_0\in\V$.

Applying the functor ${\rm Hom}_\B(-,\Sigma^k \V)$ to the $\EE$-triangle (\ref{t3}),
we have the following exact sequence:
$$0=\EE(U_0, \Sigma^k \V)\to\EE(V_0, \Sigma^k \V)\to \EE^2(N,\Sigma^k \V)\simeq \EE(N,\Sigma^{k+1} \V)=0.$$
It follows that $\EE(V_0,\Sigma^k\V)=0$, which implies $V_0\in\W_k$. Hence $N\in\Cone(\W_k,\W)=\W_{k+1}$.
\qed
\vspace{3mm}

By this lemma, $\W_i~ (i\geq 1)$ is closed under extensions, it is an extriangulated subcategory of $\B$.
\smallskip

We show the following theorem.

\begin{thm}\label{main1}
Let $(\U,\V)$ be a hereditary cotorsion pair in $\B$. Assume that
$$({^{\bot_1}}(\Sigma \V),\Sigma \V) \text{ and } ({^{\bot_1}}(\Sigma^i \V),\Sigma^i \V)$$
are also cotorsion pairs in $\B$. Then there exists a one-to-one correspondence from the first of the following sets to the second.
\begin{itemize}
\item[(a)] Cotorsion pairs $(\U',\V')$ such that $\U\subseteq \U'\subseteq {^{\bot_1}}(\Sigma^i \V)$.
\medskip

\item[(b)] Cotorsion pairs $(\X,\Y)$ in $\W_i$.
\end{itemize}
\end{thm}

\begin{proof}
We first denote the correspondences:
\begin{align*}
\Phi: (\U',\V')\mapsto (\V\cap \U',\V'\cap {^{\bot_1}}(\Sigma^i \V)), \quad
\Psi: (\X,\Y)\mapsto (\add(\U*\X),\add(\Y*\Sigma^i \V)).
\end{align*}
For a cotorsion pairs $(\U',\V')$ such that $\U\subseteq \U'\subseteq {^{\bot_1}}(\Sigma^i \V)$, we show that $(\V\cap \U',\V'\cap {^{\bot_1}}(\Sigma^i \V))$ is a cotorsion pair in $\W_i$. %Note that we do not need the assumption $({^{\bot_1}}(\Sigma \V),\Sigma \V)$ is a cotorsion pair here.

$\bullet$ Since $\U\subseteq \U'\subseteq {^{\bot_1}}(\Sigma^i \V)$, we have $\Sigma^i \V \subseteq \V'\subseteq \V$. Hence by Lemma \ref{imp}, we have $\X:=\V\cap \U'\subseteq \W_i$ and $\Y:=\V'\cap {^{\bot_1}}(\Sigma^i \V)\subseteq \W_i$. Moreover, $\EE(\X,\Y)=0$.

By Lemma \ref{2}, it suffices to show that any object $G\in \W_i$ admits an $\EE$-triangle $G\rightarrow Y\rightarrow X\dashrightarrow$ where $X\in \X$ and $Y\in \Y$.

Since $({^{\bot_1}}(\Sigma^i \V),\Sigma^i \V)$ is cotorsion pair, there exists an $\EE$-triangle
$$\Sigma^i V\rightarrow R\rightarrow G\overset{}{\dashrightarrow}$$
where $R\in{^{\bot_1}}(\Sigma^i \V)$ and $\Sigma^i V\in\Sigma^i \V$. Since $(\U',\V')$ is cotorsion pair,
there exists an $\EE$-triangle
$$R\rightarrow V'\rightarrow U'\overset{}{\dashrightarrow}$$
where $U'\in\U'$ and $V'\in\V'$.
We get the following commutative diagram made of $\EE$-triangles
$$
\xymatrix{\Sigma^i V\ar[r]\ar@{=}[d]&R\ar[r]\ar[d]&G\ar@{-->}[r]\ar[d]&\\
\Sigma^i V\ar[r]&V'\ar[d]\ar[r]&Y\ar[d]\ar@{-->}[r]&\\
&U'\ar@{-->}[d]\ar@{=}[r]&U'\ar@{-->}[d]&\\
&&&}$$
Since $\V$ is extension closed and $\Sigma^i V\in\V,G\in\V$, we obtain $R\in\V$.
%It follows that $\EE(\U,\Sigma^i U)=0$ since $\U\subseteq {^{\bot_1}}(\Sigma^i \V)$.
Applying the functor ${\rm Hom}_\B(\U,-)$ to the $\EE$-triangle $R\rightarrow V'\rightarrow U'\overset{}{\dashrightarrow}$,
we have the following exact sequence:
$$0=\EE(\U,V')\to\EE(\U,U')\to \EE^2(\U,R)=0.$$
It follows that $\EE(\U,U')=0$ implies $U'\in\V$ by Lemma \ref{y1}.
Thus $U'\in\V\cap\U'=\X$.

Since $\U'\subseteq {^{\bot_1}}(\Sigma^i \V)\subseteq {^{\bot_1}}(\Sigma^{i+1} \V)$, applying the functor ${\rm Hom}_\B(\U',-)$ to the $\EE$-triangle $V\rightarrow V'\rightarrow Y\overset{}{\dashrightarrow}$,
we have the following exact sequence:
$$0=\EE(\U',V')\to\EE(\U',Y)\to \EE^2(\U',\Sigma^i V)\simeq \EE(\U',\Sigma^{i+1} V)=0.$$
It follows that $\EE(\U',Y)=0$ implies $Y\in\V'$ by Lemma \ref{y1}.

Since $G,U'\in {^{\bot_1}}(\Sigma^i \V)$, we have $Y\in{^{\bot_1}}(\Sigma^i \V)$. Thus $Y\in\V'\cap{^{\bot_1}}(\Sigma^i \V)=\Y$.

\medskip

$\bullet$ Now let $(\X,\Y)$ be a cotorsion pair in $\W_i$. Let $\add(\U*\X)=:\U'$ and $\add(\Y*\Sigma^i \V)=:\V'$. We show that $(\U',\V')$ is a cotorsion pair such that $\U\subseteq \U'\subseteq {^{\bot_1}}(\Sigma^i \V)$.\\
%Since $\EE(\U',\Sigma \V)\simeq \EE(\Omega \U',\V)$, we have $\EE(\Omega \U',\V')=0$, then $\Omega \U'\subseteq \U'$ and by definition $(\U',\V')$ is hereditary.\\
Since $\U\subseteq {^{\bot_1}}(\Sigma^i \V)$ and $\X\subseteq {^{\bot_1}}(\Sigma^i \V)$. By the definition we have $\U\subseteq \U'\subseteq {^{\bot_1}}(\Sigma^i \V)$. By Lemma \ref{imp} and Lemma \ref{y2}, we have $\EE(\U,\Y)=0$, $\EE(\U,\Sigma^i \V)=0$ and $\EE(\X,\Sigma^i \V)=0$. Since $(\X,\Y)$ is a cotorsion pair in $\W_i$, by the definition we have $\EE(\X,\Y)=0$. Hence we have $\EE(\U',\V')=0$.\\
Let $B\in \B$ be any object. It admits an $\EE$-triangle $B\rightarrow V\rightarrow U\dashrightarrow$ where $U\in \U$ and $V\in \V$. Since $({^{\bot_1}}(\Sigma^i \V),\Sigma^i \V)$ is a cotorsion pair, $V$ admits an $\EE$-triangle $\Sigma^i V_0\rightarrow R\rightarrow V\dashrightarrow$ where $V_0\in \V$ and $R\in {^{\bot_1}}(\Sigma^i \V)$. Since $\V$ is closed under extension, $R\in \V$ and thus $R\in \W_i$. Then $R$ admits an $\EE$-triangle $R\rightarrow Y\rightarrow X\dashrightarrow$ where $X\in \X$ and $Y\in \Y$. We have the following commutative diagram
$$\xymatrix{
\Sigma^i V_0\ar@{=}[r] \ar[d] & \Sigma^i V_0\ar[d]\\
R \ar[r] \ar[d] &Y \ar[r] \ar[d] &X \ar@{=}[d] \ar@{-->}[r] &\\
V \ar[r] \ar@{-->}[d] &V'\ar[r] \ar@{-->}[d] &X \ar@{-->}[r] &\\
&&&
}
$$
We also have the following commutative diagram
$$\xymatrix{
\Sigma^i V_0 \ar[r] \ar@{=}[d] &Y\ar[d] \ar[r] &V'\ar@{-->}[r] \ar[d] &\\
\Sigma^i V_0 \ar[r] &I_0 \ar[r] &\Sigma^{i+1} V_0\ar@{-->}[r] &
}
$$
which induces an $\EE$-triangle $Y\rightarrow V'\oplus I_0\rightarrow \Sigma^{i+1} V_0\dashrightarrow$. Since $\Sigma^{i+1} V_0\in \Sigma^i \V$, we get $V'\in \V'$. In the following commutative diagram
$$\xymatrix{
B\ar[r] \ar@{=}[d] &V\ar[r] \ar[d] &U\ar@{-->}[r] \ar[d] &\\
B \ar[r] &V' \ar[r] \ar[d] &U' \ar[d] \ar@{-->}[r] &\\
&X \ar@{-->}[d] \ar@{=}[r] &X \ar@{-->}[d]\\
&&&
}
$$
we have $U'\in\U'$. Hence $B$ admits an $\EE$-triangle $B\rightarrow V'\rightarrow U'\dashrightarrow$ where $U'\in \U'$ and $V'\in \V'$. Since $({^{\bot_1}}(\Sigma \V),\Sigma \V)$ is also a cotorsion pair, $B$ admits the following commutative diagram
$$\xymatrix{
\Sigma^i V_1\ar@{=}[r] \ar[d] & \Sigma^i V_1\ar[d]\\
R_1 \ar[r] \ar[d] &\Sigma V_2 \ar[r] \ar[d] &R_2 \ar@{=}[d] \ar@{-->}[r] &\\
B \ar[r] \ar@{-->}[d] &B'\ar[r] \ar@{-->}[d] &R_2 \ar@{-->}[r] &\\
&&&
}
$$
where $V_1,V_2\in \V$, $R_1\in {^{\bot_1}}(\Sigma^i \V)$ and $R_2\in {^{\bot_1}}(\Sigma \V)$. Moreover, we have $\Sigma V_2\in \W_1\subseteq \W_i$, hence it admits an $\EE$-triangle $Y_0\rightarrow X_0\rightarrow \Sigma V_2 \dashrightarrow$ where $X_0\in \X$ and $Y_0\in \Y$. Then we have the following commutative diagrams
$$\xymatrix{
Y_0\ar@{=}[r] \ar[d] &Y_0\ar[d]\\
U'_0 \ar[r] \ar[d] &X_0 \ar[r] \ar[d] &R_2 \ar@{=}[d]  \ar@{-->}[r] &\\
R_1 \ar[r] \ar@{-->}[d] &\Sigma V_2 \ar[r] \ar@{-->}[d] &R_2 \ar@{-->}[r] &,\\
&&&\\
}
\quad \quad
\xymatrix{
\\
\Omega R_2 \ar[r] \ar[d] &P_2 \ar[r] \ar[d]&R_2 \ar@{=}[d]  \ar@{-->}[r] &\\
U'_0 \ar[r]  &X_0 \ar[r]  &R_2  \ar@{-->}[r] &\\
}
$$
where $P_2\in \mathcal P$. The second diagram induces an $\EE$-triangle $\Omega R_2 \rightarrow U'_0\oplus P_2\rightarrow X_0\dashrightarrow$. Since $\EE(\Omega R,\V)\simeq  \EE(R,\Sigma \V)=0$,
%we have the following exact sequence
%$$0=\EE(P_2,\V)\to \EE(\Omega R,\V) \to \EE^2(R_2,\V)=0,$$
we get $\Omega R\in \U$, hence $U'_0\in \U'$. Now we have the following commutative diagram
$$\xymatrix{
Y_0 \ar@{=}[r] \ar[d] &Y_0\ar[d]\\
V'_0 \ar[r] \ar[d] &U'_0 \ar[r] \ar[d] &B \ar@{=}[d]  \ar@{-->}[r] &\\
\Sigma V_1\ar[r] \ar@{-->}[d] &R_1 \ar[r] \ar@{-->}[d] &B \ar@{-->}[r] &\\
&&&
}
$$
where $V'_0\in \V'$. Hence $B$ admits an $\EE$-triangle $V'_0\rightarrow U'_0\rightarrow B\dashrightarrow$ where $U'_0\in \U'$ and $V'_0\in \V'$. By definition $(\U',\V')$ is a cotorsion pair.
\medskip

$\bullet$ We show that $\Phi\Psi=\id$ and $\Psi\Phi=\id$.\\
Let $(\X,\Y)$ be a cotorsion pair in $\W_i$. Let $(\U',\V')=\Psi((\X,\Y))$ and $(\X',\Y')=\Phi((\U',\V'))$. Since $\Y'=\V'\cap {^{\bot_1}}(\Sigma^i \V)\subseteq \V'=\add (\Y*\Sigma^i\V)$ and $\X'=\V\cap \U'=\V\cap \add(\U*\X)$, we have $\EE(\X,\Y')=0$. Hence $\Y'\subseteq \Y$. Since $\X'=\V'\cap \U$, we have $\EE(\X',\Y)=0$, hence $\X'\subseteq \X$. Then by Corollary \ref{cor1}, we have $(\X,\Y)=(\X',\Y')$. Hence $\Phi\Psi=\id$.\\[1mm]
Let $(\U',\V')$ be a cotorsion pair such that $\U\subseteq \U'\subseteq {^{\bot_1}}(\Sigma^i \V)$. Consider the cotorsion pair
 $$(\U'',\V'')=(\add(\U*(\V\cap\U')), \add(\V'\cap {^{\bot_1}}(\Sigma^i \V))*\Sigma^i \V),$$ we have $\U''\subseteq \U'$ and $\V''\subseteq \V'$. Then by Corollary \ref{cor1}, we have $(\U'',\V'')=(\U',\V')$. Hence $\Psi\Phi=\id$.
\end{proof}

\begin{rem}
The assumption
$$({^{\bot_1}}(\Sigma \V),\Sigma \V) \text{ and } ({^{\bot_1}}(\Sigma^i \V),\Sigma^i \V) \text{  are cotorsion pairs}$$
is automatically satisfied when $\B$ is a triangulated category.
\end{rem}

We give an example of our main results.

\begin{exm}
Let $Q$ be the following quiver.
$$\cdot \xrightarrow{x}\cdot \xrightarrow{x}\cdot \xrightarrow{x}\cdot \xrightarrow{x}\cdot \xrightarrow{x}\cdot \xrightarrow{x}\cdot \xrightarrow{x}\cdot \xrightarrow{x}\cdot $$
Put $A=kQ/[x^4]$. The AR-quiver of $\mod A$ is the following
$$
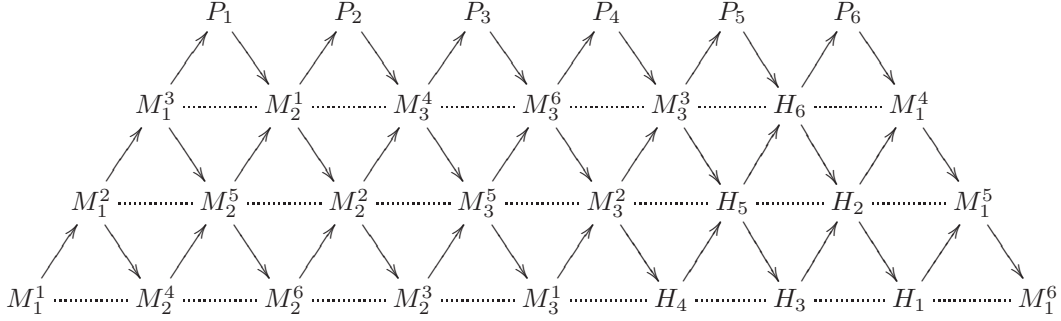
\xymatrix@C=0.1cm@R0.7cm{
&&&P_1 \ar[dr] &&P_2 \ar[dr] &&P_3 \ar[dr] &&P_4 \ar[dr] &&P_5 \ar[dr] &&P_6 \ar[dr] \\
 &&M^3_1 \ar[dr]   \ar[ur]  \ar@{.}[rr] &&M^1_2 \ar[ur]  \ar@{.}[rr]  \ar[dr]  &&M^4_3 \ar[ur]  \ar@{.}[rr] \ar[dr]   &&M^6_3 \ar[ur]  \ar@{.}[rr]  \ar[dr]  &&M^3_3 \ar[dr] \ar[ur]  \ar@{.}[rr]  &&H_6  \ar[dr]  \ar[ur]  \ar@{.}[rr]  &&M^4_1 \ar[dr]  \\
 &M^2_1 \ar@{.}[rr] \ar[dr] \ar[ur] &&M^5_2 \ar@{.}[rr] \ar[dr] \ar[ur] &&M^2_2 \ar[dr] \ar@{.}[rr] \ar[ur] &&M^5_3 \ar[dr] \ar@{.}[rr] \ar[ur] &&M^2_3 \ar[dr] \ar@{.}[rr] \ar[ur] &&H_5 \ar[dr] \ar[ur] \ar@{.}[rr] &&H_2 \ar[dr] \ar@{.}[rr] \ar[ur] &&M^5_1 \ar[dr] \\
M^1_1 \ar[ur]  \ar@{.}[rr] &&M^4_2  \ar@{.}[rr] \ar[ur] &&M^6_2 \ar@{.}[rr] \ar[ur] &&M^3_2   \ar[ur] \ar@{.}[rr] &&M^1_3 \ar@{.}[rr]  \ar[ur] &&H_4  \ar[ur] \ar@{.}[rr] &&H_3 \ar@{.}[rr]  \ar[ur] &&H_1 \ar[ur] \ar@{.}[rr] &&M^6_1
}
$$
Note that $A=\bigoplus^{3}_{i=1}M^i_1\oplus \bigoplus^{6}_{i=1}P_i$.

Pick a hereditary cotorsion pair $(\U,\V)$ of $\mod A$ where $\U=\add(A\oplus \bigoplus^{6}_{i=1}M^i_2).$ We have $\Sigma \V:$
$$\xymatrix@C=0.1cm@R0.5cm{
&&&\circ \ar[dr] &&\circ \ar[dr] &&\circ \ar[dr] &&\circ \ar[dr] &&\circ \ar[dr] &&\circ \ar[dr] \\
 &&\cdot \ar[dr]   \ar[ur]  \ar@{.}[rr] &&\cdot \ar[ur]  \ar@{.}[rr]  \ar[dr]  &&\cdot  \ar[ur]  \ar@{.}[rr] \ar[dr]   &&\cdot \ar[ur]  \ar@{.}[rr]  \ar[dr]  &&\circ  \ar[dr] \ar[ur]  \ar@{.}[rr]  &&\circ \ar[dr]  \ar[ur]  \ar@{.}[rr]  &&\circ  \ar[dr]  \\
 &\cdot \ar@{.}[rr] \ar[dr] \ar[ur] &&\cdot \ar@{.}[rr] \ar[dr] \ar[ur] &&\cdot \ar[dr] \ar@{.}[rr] \ar[ur] &&\cdot  \ar[dr] \ar@{.}[rr] \ar[ur] &&\circ  \ar[dr] \ar@{.}[rr] \ar[ur] &&\circ  \ar[dr] \ar[ur] \ar@{.}[rr] &&\circ  \ar[dr] \ar@{.}[rr] \ar[ur] &&\circ  \ar[dr] \\
\cdot \ar[ur]  \ar@{.}[rr] &&\cdot \ar@{.}[rr] \ar[ur] &&\cdot \ar@{.}[rr] \ar[ur] &&\cdot   \ar[ur] \ar@{.}[rr] &&\circ  \ar@{.}[rr]  \ar[ur] &&\circ   \ar[ur] \ar@{.}[rr] &&\circ  \ar@{.}[rr]  \ar[ur] &&\circ  \ar[ur] \ar@{.}[rr] &&\circ
}
$$
and $\Sigma^2 \V:$
$$\xymatrix@C=0.1cm@R0.5cm{
&&&\circ \ar[dr] &&\circ \ar[dr] &&\circ \ar[dr] &&\circ \ar[dr] &&\circ \ar[dr] &&\circ \ar[dr] \\
 &&\cdot \ar[dr]   \ar[ur]  \ar@{.}[rr] &&\cdot \ar[ur]  \ar@{.}[rr]  \ar[dr]  &&\cdot  \ar[ur]  \ar@{.}[rr] \ar[dr]   &&\cdot \ar[ur]  \ar@{.}[rr]  \ar[dr]  &&\cdot  \ar[dr] \ar[ur]  \ar@{.}[rr]  &&\circ \ar[dr]  \ar[ur]  \ar@{.}[rr]  &&\circ  \ar[dr]  \\
 &\cdot \ar@{.}[rr] \ar[dr] \ar[ur] &&\cdot \ar@{.}[rr] \ar[dr] \ar[ur] &&\cdot \ar[dr] \ar@{.}[rr] \ar[ur] &&\cdot  \ar[dr] \ar@{.}[rr] \ar[ur] &&\cdot  \ar[dr] \ar@{.}[rr] \ar[ur] &&\cdot  \ar[dr] \ar[ur] \ar@{.}[rr] &&\circ  \ar[dr] \ar@{.}[rr] \ar[ur] &&\circ  \ar[dr] \\
\cdot \ar[ur]  \ar@{.}[rr] &&\cdot \ar@{.}[rr] \ar[ur] &&\cdot \ar@{.}[rr] \ar[ur] &&\cdot   \ar[ur] \ar@{.}[rr] &&\cdot  \ar@{.}[rr]  \ar[ur] &&\cdot   \ar[ur] \ar@{.}[rr] &&\cdot  \ar@{.}[rr]  \ar[ur] &&\circ  \ar[ur] \ar@{.}[rr] &&\circ
}
$$
They are extension closed, hence $({^{\bot_1}}(\Sigma \V),\Sigma \V) \text{ and } ({^{\bot_1}}(\Sigma^2 \V),\Sigma^2 \V)$ are cotorsion pairs.
We have $\W=\add(\bigoplus^{3}_{i=1}M^i_2\oplus \bigoplus^{6}_{i=1}P_i)$%and $\G:$
%$$\xymatrix@C=0.1cm@R0.5cm{
%&&&\circ \ar[dr] &&\circ \ar[dr] &&\circ \ar[dr] &&\circ \ar[dr] &&\circ \ar[dr] &&\circ \ar[dr] \\
 %&&\cdot \ar[dr]   \ar[ur]  \ar@{.}[rr] &&\circ \ar[ur]  \ar@{.}[rr]  \ar[dr]  &&\circ  \ar[ur]  \ar@{.}[rr] \ar[dr]   &&\circ \ar[ur]  \ar@{.}[rr]  \ar[dr]  &&\circ  \ar[dr] \ar[ur]  \ar@{.}[rr]  &&\cdot \ar[dr]  \ar[ur]  \ar@{.}[rr]  &&\cdot  \ar[dr]  \\
 %&\cdot \ar@{.}[rr] \ar[dr] \ar[ur] &&\cdot \ar@{.}[rr] \ar[dr] \ar[ur] &&\circ \ar[dr] \ar@{.}[rr] \ar[ur] &&\circ  \ar[dr] \ar@{.}[rr] \ar[ur] &&\circ  \ar[dr] \ar@{.}[rr] \ar[ur] &&\cdot  \ar[dr] \ar[ur] \ar@{.}[rr] &&\cdot  \ar[dr] \ar@{.}[rr] \ar[ur] &&\cdot  \ar[dr] \\
%\cdot \ar[ur]  \ar@{.}[rr] &&\cdot \ar@{.}[rr] \ar[ur] &&\cdot \ar@{.}[rr] \ar[ur] &&\circ   \ar[ur] \ar@{.}[rr] &&\circ  \ar@{.}[rr]  \ar[ur] &&\cdot  \ar[ur] \ar@{.}[rr] &&\cdot  \ar@{.}[rr]  \ar[ur] &&\cdot  \ar[ur] \ar@{.}[rr] &&\cdot
%}
%$$
and indecomposable of objects in $\G-\W$ are $\{ M^i_3, i=1,2,...,6  \}$. Let $\U'=\add(A\oplus \bigoplus^{6}_{i=1}M^i_2\oplus \bigoplus^{6}_{i=4}M^i_3)$. We have $\U\subseteq \U'\subseteq {^{\bot_1}}(\Sigma \V)\subseteq  {^{\bot_1}}(\Sigma^2 \V)$. Then by Theorem \ref{main1} we can get a cotorsion pair $(\X,\Y_1)$ in $\G$ (resp. $(\X,\Y_2)$ in $\W_2$) where then indecomposable of objects in $\X-\W$ are $\{ M^4_3, M^5_3, M^6_3\}$.

%Let $\G_{\M}=\Cone(\M,\M)$. Then $\G_{\M}=\add(\bigoplus^{6}_{i=1}M^i_1\oplus \bigoplus^{6}_{i=1}M^i_2\oplus \bigoplus^{6}_{i=1}P_i)$ and $\overline \G_{\M}=\add(\bigoplus^{3}_{i=1}M^i_1 \oplus \bigoplus^{6}_{i=4}M^i_2)$. Pick a support $\tau$-titling subcategory $\N=\add(M^3_1\oplus M^4_2\oplus M^5_2)$, by Theorem \ref{main2} we can get a cotorsion piar $(\X_0,\Y_0)$ in $\G_{\M}$ where $\X_0=\add(\bigoplus^{6}_{i=1} M^i_1\oplus \bigoplus^{5}_{i=4}M^i_2\oplus \bigoplus^{6}_{i=1}P_i)$ and $\Y_0=\add(\bigoplus^{6}_{i=3} M^i_1 \oplus\bigoplus^{6}_{i=1}M^i_2\oplus \bigoplus^{6}_{i=1}P_i)$.
\end{exm}

%When $\B$ is triangulated and $(\U,\V)$ is a cotorsion pair induced by a silting subcategory $\W$, we have $\bigcup_{i\geq 0} \W_i=\V$ and $\bigcap_{i\geq 1} \Sigma^i \V=0$. Then we have the following corollary

%\begin{prop}
%If $\B$ has finite projective (resp. injective, global) dimension, then the following correspondences:
%\begin{align*}
%\Phi: (\U',\V')\mapsto (\V\cap \U',\V'), \quad
%\Psi: (\X,\Y)\mapsto (\add(\U*\X),\Y).
%\end{align*}
%are one-to-one correspondences between the first of the following sets and the second.
%\begin{itemize}
%\item[(a)] Cotorsion pairs $(\U',\V')$ such that $\U\subseteq \U'$.
%\medskip

%\item[(b)] Cotorsion pairs $(\X,\Y)$ in $ \V$.
%\end{itemize}
%\end{prop}

%\begin{proof}
%By Lemma \ref{dim}, we have$\bigcup_{i\geq 1} \Sigma^i \V=\mathcal I$, hence $\bigcup_{i\geq 0} \W_i$.It is followed by the fact that
%\end{proof}

\section{A study on proper $m$-term subcategories}

In this section, let $\W=\U\cap \V$ where $(\U,\V)$ is a hereditary cotorsion pair. We will use the notions associated with $\W$ in Section 4. We also assume $\mathcal I\subsetneq \W$.

\begin{prop}\label{contra}
$\W_i$ is contravariantly finite in $\W_{j}$ when $i<j$.
\end{prop}

\begin{proof}
It is enoguh to show that $\W_i$ is contravariantly finite in $\W_{i+1}$. The case $i=0$ is trivial.

Assume that we have shown $\W_{i-1}$ is contravariantly finite in $\W_i$. We show that $\W_i$ is contravariantly finite in $\W_{i+1}$.

Let $A\in \W_{i+1}$. Then we have two $\EE$-triangles
$$\xymatrix{A_1 \ar[r]^{a_1} &W_1 \ar[r]^{w_1} &A \ar@{-->}[r]  &}~~\mbox{and}~~ \xymatrix{C \ar[r] &A_2 \ar[r]^{a_2} &A_1 \ar@{-->}[r]  &}$$
where $W_1\in W$, $A_1 \in \W_i$ and $w_2$ is a right minimal $\W_{i-1}$-approximation. Then $C\in \W_{i-1}^{\bot_1}$. We can get a commutative diagram
$$\xymatrix{
A_2 \ar[r]^{i_2}  \ar[d]_{a_2} &I_2 \ar[r] \ar[d] &\Sigma A_2 \ar@{=}[d] \ar@{-->}[r]  &\\
A_1 \ar[d]^{a_1} \ar[r] &\Sigma C \ar[r] \ar[d] &\Sigma A_2 \ar@{=}[d] \ar@{-->}[r] &\\
W_1 \ar[r]_u \ar[d] &B \ar[r] \ar[d]^b &\Sigma A_2 \ar@{-->}[r]  &\\
A  \ar@{-->}[d] \ar@{=}[r] &A  \ar@{-->}[d]\\
&&
}
$$
Since $W_1, \Sigma A_2\in \W_i$, we have $B\in \W_i$. Let $D\in \W_i$. It admits an $\EE$-triangle $D_1 \xrightarrow{} W_1'\to D\dashrightarrow$ where $W_1'\in \W$ and $D_1\in \W_{i-1}$. Let $d: D\to A$ be any morphism. We have the following commutative diagram:
$$\xymatrix{
 D_1\ar[r]^{d_1} \ar[d]^c &W_1' \ar[r] \ar[d] &D \ar[d]^{d} \ar@{-->}[r]  &\\
 \Sigma C \ar[r] &B \ar[r]^b &A \ar@{-->}[r]  &
 }
 $$
Since $c$ factors through an object in $\mathcal I$, it factors through $d_1$, hence $d$ factors through $b$. Thus $\W_i$ is contravriantly finite in $\W_{i+1}$.
\end{proof}

By the same method, we can get the following corollary when $\B$ is a triangulated category.

\begin{cor}
If $\W$ is  contravariantly finite, then for any $i>0$, $\W_i$ is also contravariantly finite.
\end{cor}

\begin{lem}\label{enough2}
$\W_i$ ($i>0$) has enough projectives $\W$ and enough injectives $\Sigma^i \W$.
\end{lem}

\begin{proof}
Since $\W_i=\V\cap{^{\bot_1}\Sigma^i \V}$, we can easily get that $\W$ is the enough projecitves in $\W_i$, and $\Sigma^i \W$ is injective in $\W_i$. We only need to show that $\Sigma^i \W$ is enough.

We have already shown that the property holds when $i=1$. Assume that we have shown that the property holds when $i=k$, we prove that it still holds when $i=k+1$.

Let $A\in \W_{k+1}$. Then we have two $\EE$-triangles
$$\xymatrix{A_1 \ar[r]^{a_1} &W_1 \ar[r]^{w_1} &A \ar@{-->}[r]  &}~~\mbox{and}~~\xymatrix{A_1 \ar[r] &\Sigma^k W \ar[r] &B \ar@{-->}[r]  &}$$
where where $W,W_1\in W$, $A_1,B \in \W_k$. We have the following commutative diagram
$$\xymatrix{
A_1 \ar[r]^{a_1} \ar@{=}[d] &W_1 \ar[d] \ar[r]^{w_1} &A \ar@{-->}[r] \ar[d]^a &\\
A_1 \ar[r] &I_1 \ar[d] \ar[r] &\Sigma A_1 \ar[d] \ar@{-->}[r] &\\
&\Sigma W_1 \ar@{-->}[d] \ar@{=}[r] &\Sigma W_1 \ar@{-->}[d]\\
&&
}
$$
where $I_1 \in \mathcal I$.  We also have the following commutative diagram
$$\xymatrix{
A_1 \ar[r] \ar@{=}[d] &\Sigma^k W \ar[d] \ar[r] &B \ar@{-->}[r]  \ar[d] &\\
A_1 \ar[r] &I_1 \ar[r]  &\Sigma A_1 \ar@{-->}[r] &
}
$$
which induces an $\EE$-triangle $\Sigma^k W \to I_1\oplus B \to \Sigma A_1\dashrightarrow$. Then we can get the following commutative diagram
$$\xymatrix{
\Sigma^k W \ar[r] \ar@{=}[d] &I_1\oplus B \ar[r] \ar[d] &\Sigma A_1 \ar[d] \ar@{-->}[r] &\\
\Sigma^k W \ar[r] &I_1\oplus I_B \ar[r] \ar[d] &\Sigma^{k+1} W \ar[d] \ar@{-->}[r] &\\
&\Sigma B \ar@{=}[r] \ar@{-->}[d] &\Sigma B \ar@{-->}[d] &\\
&&
}
$$
Now we get the following commutative diagram
$$\xymatrix{
A \ar[r] \ar@{=}[d] &\Sigma A_1 \ar[r] \ar[d] &\Sigma W_1 \ar[d] \ar@{-->}[r] &\\
A \ar[r] &\Sigma^{k+1} W \ar[r] \ar[d] &B_1 \ar[d] \ar@{-->}[r] &\\
&\Sigma B \ar@{=}[r] \ar@{-->}[d] &\Sigma B\ar@{-->}[d]\\
&&
}
$$
Since $\Sigma W_1, \Sigma B\in \W_{k+1}$, we have $B_1\in \W_{k+1}$. Hence $\Sigma^{k+1} \W$ is the enough injectives in $\W_{k+1}$.
\end{proof}

\begin{lem}
$\EE^{m}(\W_i,\W_j)=0$ when $i\leq j$ and $i<m$.
\end{lem}

\begin{proof}
We first show the case $i=j=1$.

Let $A,B\in \W_1$. It admits an $\EE$-triangle $W_1\to W_0\to A\dashrightarrow$ where $W_1,W_0\in\W$. Then we have the following exact sequence $0=\EE^m(B,W_0)\to \EE^m(B,A)\to \EE^{m+1}(B,W_1)=0$ when $m\geq 2$, which implies $\EE^m(\W_1,\W_1)=0$.

Now we assume that the property holds true when $i=1$ and $j=k$. Let $j=k+1$, $A\in \W_j$ and $B\in \W_1$. $A$ admits an $\EE$-triangle $A_1\to W_0\to A\dashrightarrow$ where $W_0\in\W$ and $A_1\in \W_k$. Then we have the following exact sequence $0=\EE^{m}(B,W_0)\to \EE^{m}(B,A)\to \EE^{m+1}(B,A_1)=0$, which implies $\EE^m(\W_1,\W_{k+1})=0$.

Now we assume that the property holds true when $i=k$ and $j\geq i$. Let $i=j=k+1$, $A,B\in \W_{k+1}$.  $A$ admits an $\EE$-triangle $A_1\to W_0\to A\dashrightarrow$ where $W_0\in\W$ and $A_1\in \W_k$. Then we have the following exact sequence $0= \EE^{m}(A_1,B) \to \EE^{m+1} (A,B)\to \EE^{m+1}(W_0,B)=0$ when $m>k$, hence $\EE^m(\W_i,\W_i)=0$ when $m>k+1$.

Now we assume that the property holds true when $i=k+1$ and $j=k'\geq i$. Let $j=k'+1$, $A\in \W_j$ and $B\in \W_i$. $A$ admits an $\EE$-triangle $A_1\to W_0\to A\dashrightarrow$ where $W_0\in\W$ and $A_1\in \W_{k'}$. Then we have the following exact sequence $0=\EE^{m}(B,W_0)\to \EE^{m}(B,A)\to \EE^{m+1}(B,A_1)=0$ when $m>i$, which implies $\EE^m(\W_i,\W_{k'+1})=0$.

Hence we get that $\EE^{m}(\W_i,\W_j)=0$ when $i\leq j$, $i<m$.
\end{proof}

By the similar method, we can show the following lemma.

\begin{lem}
$\EE^{m}(\W_i,\W_j)=0$ when $j< i$, $j+1<m$.
\end{lem}

By the two lemmas above, we have the following corollary.

\begin{cor}\label{higher}
$\EE^{i+1}(\W_i, \W_{j-i-1})=0$ when $i<j$.
\end{cor}

%\begin{proof}
%Let $j=1$ and $i=2$, $A\in \W_2$ and $B\in \W_1$. $A$ admits an $\EE$-triangle $A_1\to W_0\to A\dashrightarrow$ where $W_0\in\W$ and $A_1\in \W_1$. Then we have the following exact sequence $0=\EE^{m-1}(A_1,B)\to \EE^{m}(A,B) \EE^{m}(W,B)=0$ when $m>2$.\\
%Now we assume that the property holds true when $j=1$ and $i=k$. Let $i=k+1$, $A\in \W_i$ and $B\in \W_1$. $A$ admits an $\EE$-triangle $A_1\to W_0\to A\dashrightarrow$ where $W_0\in\W$ and $A_1\in \W_k$. Then we have the following exact sequence $0=\EE^{m-1}(A_1,B)\to \EE^{m}(A,B)\to \EE^{m}(W_0,B)=0$ when $m>2$, which implies $\EE^m(\W_{k+1},\W_1)=0$.\\
%\end{proof}

By definition, one can show the following lemma, the proof is left to the readers.

\begin{lem}\label{summand2}
$\Sigma^i\W_j$ is closed under direct summands for $i>0$ and $j\geq 0$.
\end{lem}

Now we show the main theorem in the section.

\begin{thm}\label{main5}
$(\W_i, \Sigma^i\W_{j-i-1})$ is a hereditary cotorsion pair in $\W_j$ when $0<i<j$.
\end{thm}

\begin{proof}
We first show that $(\W_i, \Sigma^i \W)$ is a cotorsion pair in $\W_{i+1}$.\\
When $i=1$, for any object $A\in \W_2$, by the proof of Proposition \ref{contra}, there is an $\EE$-triangle $\Sigma W\to B\to A\dashrightarrow$ where $W\in \W$ and $B\in \W_1$. By Proposition \ref{enough2} and the dual of Lemma \ref{2}, we get that $(\W_1, \Sigma \W)$ is a hereditary cotorsion pair in $\W_2$.\\
Assume that we have shown $(\W_i, \Sigma^i \W)$ is a hereditary cotorsion pair in $\W_{i+1}$ when $i<k+1$. Let $i=k+1$ and $A\in \W_{k+2}$. We have two $\EE$-triangles
$$\xymatrix{A_1 \ar[r]^{a_1} &W_1 \ar[r]^{w_1} &A \ar@{-->}[r]  &,} \quad \xymatrix{\Sigma^k W \ar[r] &A_2 \ar[r] &A_1 \ar@{-->}[r]  &}$$
where $W_1,W\in \W$, $A_1 \in \W_{k+1}$ and $A_2\in \W_k$. Again by the proof Proposition \ref{contra}, we can get an $\EE$-triangle $\Sigma^{k+1} W\to B\to A\dashrightarrow$ where $B\in \W_{k+1}$.\\
Assume that we have shown $(\W_i, \Sigma^i\W_{k})$ is a hereditary cotorsion pair in $\W_{i+k+1}$, we show that $(\W_i, \Sigma^i\W_{k+1})$ is a hereditary cotorsion pair in $\W_{i+k+2}$.\\
Let $A\in \W_{i+k+2}$. We have the following commutative diagram
$$\xymatrix{
\Sigma^i B\ar@{=}[r] \ar[d] &\Sigma^i B \ar[d]\\
A_2 \ar[r] \ar[d] &A_1 \ar[r] \ar[d] &A \ar@{=}[d] \ar@{-->}[r] &\\
\Sigma^{i+k+1} W_0 \ar[r] \ar@{-->}[d] &A_0 \ar@{-->}[d] \ar[r] &A \ar@{-->}[r] &\\
&&
}
$$
where $W_0,W_1\in \W$, $A_0\in \W_{i+k+1}$, $A_1\in \W_i$ and $B\in W_k$. We have the following commutative diagrams
$$\xymatrix{
&\Sigma^{i+k} W_0 \ar@{=}[r] \ar[d] &\Sigma^{i+k} W_0 \ar[d]\\
\Sigma^i B \ar@{=}[d] \ar[r] &\Sigma^i B\oplus I_0 \ar[r] \ar[d] &I_0 \ar@{-->}[r] \ar[d] &\\
\Sigma^i B  \ar[r] &A_2 \ar[r] \ar@{-->}[d] &\Sigma^{i+k+1} W_0 \ar@{-->}[d] \ar@{-->}[r] &\\
&&
}\\ \quad \xymatrix{
\Sigma^{i-1} B \ar@{=}[r] \ar[d] &\Sigma^{i-1} B \ar[d]\\
A_3 \ar[r] \ar[d] & I_1\oplus I_0 \ar[r] \ar[d] &A_2 \ar@{=}[d] \ar@{-->}[r]&\\
\Sigma^{i+k} W_0 \ar[r] \ar@{-->}[d] & \Sigma^i B\oplus I_0 \ar[r] \ar@{-->}[d] &A_2 \ar@{-->}[r]&\\
&&
}
$$
Continue this process, finally we can get the following commutative diagrams
$$\xymatrix{
&\Sigma^{k+1} W_0 \ar@{=}[r] \ar[d] &\Sigma^{k+1} W_0 \ar[d]\\
\Sigma B \ar@{=}[d] \ar[r] &\Sigma B\oplus I_0' \ar[r] \ar[d] &I_0' \ar@{-->}[r] \ar[d] &\\
\Sigma B  \ar[r] &A_{i+1} \ar[r] \ar@{-->}[d] &\Sigma^{k+2} W_0 \ar@{-->}[d] \ar@{-->}[r] &\\
&&
}\\ \quad \xymatrix{
B \ar@{=}[r] \ar[d] & B \ar[d]\\
A_{i+2} \ar[r] \ar[d] & I_1'\oplus I_0 '\ar[r] \ar[d] &A_{i+1} \ar@{=}[d] \ar@{-->}[r]&\\
\Sigma^{k+1} W_0 \ar[r] \ar@{-->}[d] & \Sigma B\oplus I_0' \ar[r] \ar@{-->}[d] &A_{i+1} \ar@{-->}[r]&\\
&&
}
$$
Since $B,\Sigma^{k+1} W_0\in \W_{k+1}$, we have $A_{i+2}\in \W_{k+1}$. Hence $A_2 \in \Sigma^i \W_{i+k+1}$. Now by Lemma \ref{summand2}, Corollary \ref{higher}, Proposition \ref{enough2} and the dual of Lemma \ref{2},  we get that $(\W_i, \Sigma^i\W_{j-i-1})$ is a hereditary cotorsion pair in $\W_j$ when $0<i<j$.
%Now we can show the general case.\\
%By Proposition \ref{contra} and the dual of Lemma \ref{Wa}, any object $A\in \W_j$ admits an $\EE$-triangle $B \to A_i \xrightarrow{a} A \dashrightarrow$ where $a$ is a minimal right $\W_i$-approximation and $B\in \W_i^{\bot_1}$. By Proposition \ref{enough2} and the dual of Lemma \ref{2}, to get that $(\W_i, \W_i^{\bot_1}\cap \W_j)$ is a hereditary cotorsion pair in $\W_j$, it is enough to show that $B\in \W_j$.\\
%Since we have an exact sequence $0=\EE(A_i, \Sigma^j \W)\to \EE(B,\Sigma^j \W) \to \EE^2(A, \Sigma^j \W)=0$, $B\in {^{\bot_1}}(\Sigma^j \W)$. By applying $\Hom_{\B}(\U,-)$ to $B \to A_i \xrightarrow{a} A \dashrightarrow$, we get an exact $\Hom_{\B}{\U,A_i} \xrightarrow{\Hom_{\B}(\U,a)} \Hom_{\B}(\U,A) \to \EE(\U,B_1)\to \EE(\U,A_i)=0$. To get that $B_1\in \V$, we need to show that $\Hom_{\B}(\U,a)$ is surjective. The morphism $a$ can be constructed in the following way:
%$$\Sigma^{j-1} V_{j-1} \to A_{j-1} \xrightarrow{a_{j-1}} A \dashrightarrow, \Sigma^{j-2} V_{j-2} \to A_{j-2} \xrightarrow{a_{j-2}} A_{j-1} \dashrightarrow, \cdot\cdot\cdot ,\Sigma^{i} V_{i} \to A_{i} \xrightarrow{a_{i}} A_{i+1} \dashrightarrow$$
%In the sequences above, $V_{k}\in \V$, $A_{k}\in \W_{k}$ and $a=a_i...a_{j-2}a_{j-1}$. Any morphism from an object $U\in \U$ to $A$ factors through $a_{j-1}$, then factors through $a_{j-2}a_{j-1}$,..., at last factors through $a$. Hence $\Hom_{\B}(\U,a)$ is surjective and $\B\in \W_j$.
\end{proof}

%\begin{rem}
%All the results in this section works if $\B$ is an exact category and $\W$ is $m$-rigid where $m$ is large enough.
%\end{rem}

\end{document}